\documentclass[12pt]{amsart}
\usepackage{cases}
\usepackage{amsmath}
\usepackage{mathrsfs}
\usepackage{amssymb}
\usepackage{amsfonts}
\usepackage{graphicx}
\usepackage{wrapfig}
\usepackage{stmaryrd}
\usepackage{abstract}
\usepackage{upgreek}
\usepackage{pdfsync}

\usepackage{geometry}
\newgeometry{left=1.8cm,right=1.8cm,top=2cm,bottom=1.8cm}

\newtheorem{theorem}{Theorem}[section]
\newtheorem{proposition}[theorem]{Proposition}
\newtheorem{lemma}[theorem]{Lemma}
\newtheorem{corollary}[theorem]{Corollary}

\newtheorem{definition}[theorem]{Definition}
\newtheorem{example}[theorem]{Example}

\newtheorem{asmp}[theorem]{Assumption}
\newtheorem{remark}[theorem]{Remark}
\numberwithin{equation}{section}
\numberwithin{case}{section}
\numberwithin{subcase}{case}
\numberwithin{subsubcase}{subcase}

\newcommand{\bm}[1]{{\mbox{\boldmath$#1$}}}

\begin{document}
\title[Viscosity Characterization of Arbitrage Function under Model Uncertainty]{Viscosity Characterization of the Arbitrage Function under Model Uncertainty}

\date{\today}

\author{Yinghui Wang}\email{yinghui@math.columbia.edu}
\address{Department of Mathematics, Columbia University, New York, NY 10027}

\subjclass[2010]{}


\thanks{{\it Acknowledgments:} The author is greatly indebted to her advisor Professor Ioannis Karatzas for suggesting this problem, for very careful readings of the previous versions of the paper, and for invaluable advice. She is also grateful to Professor Marcel Nutz, for advice on the Dynamic Programming Principle and on the literature on the subject of this paper. This research was supported in part by the National Science Foundation under Grant NSF-DMS-14-05210.}

\maketitle

\begin{abstract}

$ $

We show that in an equity market model with  Knightian   uncertainty  regarding the relative risk and covariance structure of its assets, the {\it arbitrage function} -- defined as the reciprocal of the highest return on investment that can be achieved relative to the market using nonanticipative strategies, and under any admissible market model configuration -- is a viscosity solution of an associated \textsc{Hamilton-Jacobi-Bellman} (HJB) equation under appropriate boundedness, continuity and Markovian assumptions on the uncertainty structure.  
This result generalizes that of \textsc{Fernholz \& Karatzas} (2011), who characterized this arbitrage function as a classical solution of a  \textsc{Cauchy} problem for this HJB equation under much stronger conditions than those needed here.
\end{abstract}
 
\section{Introduction}
 \label{section:intro}

We consider an equity market with asset capitalizations $\mathfrak{X}(t) =
(X_1(t), . . . ,X_n(t))' \in (0,\infty)^n$ at time $t \in [0,\infty)$, and with local covariation rates \label{p:i}$\alpha( t,\mathfrak{X})  = \left( \alpha_{ij} (t,\mathfrak{X})\right) _{1\le i,j\le n}$ and local relative risk rates $\vartheta(t,\mathfrak{X}) = \left( \vartheta_1(t,\mathfrak{X}), . . . , \vartheta_n(t,\mathfrak{X})\right)'$,   which are nonanticipative functionals of (i.e., are determined by) the past and present capitalizations   for any given time $t$. We denote by \label{p:S+}$\mathbb{S}_{+}(n)$ is the space of real, symmetric and positive-definite  $n \times n$ matrices, fix a collection $\{\mathcal{K}(y)\}_{y\in (0,\infty) \}}$ of nonempty, compact and convex subsets on $\mathbb{R}^n \times \mathbb{S}_{+}(n)$, and pose the following question: 

\smallskip
{\it If  the   pair $\left( \vartheta(t,\mathfrak{X}),\alpha(t,\mathfrak{X})\right) $ is  restricted  to take values in   a given nonempty  subset $\mathcal{K}\left( \mathfrak{X}(t)\right) $ of $  \mathbb{R}^n \times \mathbb{S}_{+}(n) $,    what is the highest return on investment relative
to the market portfolio  over the given time horizon $[0,T]$\label{p:T}, that can be achieved using nonanticipative investment rules, when starting with initial capitalizations \label{p:x}$x=(x_1,\dots,x_n)'\in (0,\infty)^n$, and
with probability one under all possible market model configurations with the above covariance and relative risk structure?  }

Equivalently, if the initial configuration of asset capitalizations is $x=(x_1,\dots,x_n)$,  what is the smallest
proportion of the initial total market capitalization $x_1+\dots+x_n$\,, starting with which one can match or outperform the market capitalization over a given time horizon $[0,T]$, by using nonanticipative investment rules,  and
with probability one under all possible market model configurations with the above covariance and relative risk structure? 

\smallskip
Our main result offers the following answers to these two questions: $1/\mathfrak{u}(T,x)$ and $\mathfrak{u}(T,x)$, respectively. Here the function $\mathfrak{u}: [0,\infty) \times \mathbb{R}_+^n\to(0, 1]$ is, subject to appropriate conditions that will be specified as we progress,     a {\it viscosity solution} to the \textsc{Cauchy} problem for the   {\it H{\scriptsize AMILTON}-J{\scriptsize ACOBI}-B{\scriptsize ELLMAN} (HJB)} fully nonlinear  partial differential equation 
\begin{equation}\label{eq:PDE}
\big(u_t-\widehat{\mathcal{L}}u\big)(t,x) = 0\,,\quad (t,x) \in (0, \infty)\times \mathbb{R}_{+}^n 
\end{equation}
of parabolic type,  subject to the initial condition
\begin{equation}\label{eq:initial cond}
u(0\,,\cdot) = 1\,,\quad x \in  \mathbb{R}_{+}^n\,.
\end{equation} 
Here   we are using the notation 
\begin{equation}\label{eq:hatL}
\widehat{\mathcal{L}} u(t,x)\,:=\, \sup_{a\in \mathcal{A}(x)} \mathcal{L}_{a}u(t,x)\,,\quad  \mathcal{L}_{a}u(t,x)\,
:=\,\sum_{i,j}  x_i x_j a_{ij} \left(\frac{D^2_{ij}}{2} + \frac{D_i}{||x||_1}\right) u(t,x)\,,
\end{equation}
for $(t,x) \in (0, \infty)\times \mathbb{R}_{+}^n$ with $a=(a_{ij})_{1\le i,j\le n}$\,; we are also using the $\,\ell^1$-norm\label{p:l1} $\,||x||_1:=\sum_i x_i\,$,
\begin{equation}\label{eq:A}
\mathcal{A}(x):= \big\{a\in  \mathbb{S}_{+}(n): \exists\ \theta \in  \mathbb{R}^n\ \text{s.t.}\ (\theta, a) \in \mathcal{K}(x)\big\}
\end{equation}
and employ the notation   $D_{i}u=u_{x_i}$\label{p:D_i}\,, $D^2_{ij}u=u_{x_i x_j}\,$, and  
\label{p:R_+^n}$\,\mathbb{R}_{+}^n:=(0,\infty)^n$. 
Furthermore, the above function $\mathfrak{u}$ is dominated by any nonnegative classical  supersolution of this \textsc{Cauchy} problem;   thus, it is the smallest nonnegative classical supersolution of this \textsc{Cauchy} problem, whenever it is of class $\, C\left( [0,\infty) \times\mathbb{R}_{+}^n \right) \, \cap\, C^{1,2}\left( (0,\infty) \times\mathbb{R}_{+}^n \right)$.

\smallskip

The function $\mathfrak{u}$ is called the {\it arbitrage function} for a model with uncertainty, in the terminology of \cite[Sections 1 and 4]{FK11}; this   extends the arbitrage function $\mathfrak{u}_{\mathcal{M}}$ for a specified model $\mathcal{M}$ in the terminology of \cite[Section 6]{FK10}. In \cite{FK11} the authors characterized the arbitrage  function $\mathfrak{u}$  as a classical solution of the HJB equation \eqref{eq:PDE}, subject to the initial condition of \eqref{eq:initial cond}, but under much stronger assumptions on the uncertainty structure; see Theorem \ref{thm:Knightian} below. 

\smallskip
Under much weaker conditions than in \cite{FK11},  we develop here a different  characterization of the arbitrage  function $\mathfrak{u}$,  as a viscosity solution to the \textsc{Cauchy} problem of \eqref{eq:PDE}, \eqref{eq:initial cond}. We first prove in Theorems \ref{thm:viscosity1} and \ref{thm:viscosity2} that the function $\widehat{\Phi}$ -- defined as the supremum of $\mathfrak{u}_{\mathcal{M}}$ over all possible market models $\mathcal{M}$ that satisfies certain strong Markov property ({\it strongly Markovian admissible systems} in Definition \ref{def:MM}) -- and the function ${\Phi}$ --
defined as the supremum of $\mathfrak{u}_{\mathcal{M}}$ over all possible market models $\mathcal{M}$ -- are viscosity subsolution and viscosity supersolution of this \textsc{Cauchy} problem, respectively.

Moreover, we show in Theorem \ref{thm:u=Phi} that the function $\mathfrak{u}$ coincides with  ${\Phi}$,  if this latter  function is continuous. As a consequence, the function $\mathfrak{u}$ is shown to be a viscosity supersolution of (\ref{eq:PDE}), and further, a viscosity solution of (\ref{eq:PDE}) if $\Phi\equiv \widehat{\Phi}^*$ (the {\em upper-semicontinuous envelope} of $\Phi$\,; see \eqref{eq:u^*}).

\subsection{Preview}

Section \ref{section:notation}   sets up the model for an equity market with model uncertainty regarding its covariance and relative risk characteristics, and Section \ref{section:previous} interprets the variables in this model,  introduces the concepts of investment rules and portfolios as well as the notion of arbitrage function, and reviews the results of \cite{FK11}. 

Section \ref{section:visc} recalls the definition of viscosity solutions, states our main results and discusses related work. Section \ref{section:subsol} characterizes the function $\widehat{\Phi}$ as a viscosity subsolution --  and further, in Section \ref{section:supersol}, the function ${\Phi}$ as a viscosity supersolution -- to the \textsc{Cauchy} problem of \eqref{eq:PDE}, \eqref{eq:initial cond}. 
 
Section \ref{section:u} provides conditions, under which the arbitrage function $\mathfrak{u}$ coincides with the function ${\Phi}$   (Theorems \ref{thm:min} and  \ref{thm:u=Phi}), and thus becomes a viscosity solution to the \textsc{Cauchy} problem \eqref{eq:PDE}, \eqref{eq:initial cond}.  Furthermore, these conditions imply that, if $\mathfrak{u}$ is of class $C\left( [0,\infty) \times\mathbb{R}_{+}^n \right) \, \cap\, C^{1,2}\left( (0,\infty) \times\mathbb{R}_{+}^n \right) $,   it is a classical solution and in fact the smallest nonnegative (super)solution of this \textsc{Cauchy} problem (Corollary \ref{coro:u sol}). Additional results, namely,   Propositions \ref{prop:suff} and  \ref{prop:u_M supersol}, provide conditions on the covariance and relative risk structure, under which $\mathfrak{u}\equiv{\Phi}\equiv\widehat{\Phi}$ and it is indeed the smallest nonnegative (super)solution of this \textsc{Cauchy} problem.

\smallskip
Section  \ref{appendix:min} develops the proof of Theorem \ref{thm:min}. 
Section \ref{section:example} concludes with examples from the (generalized) volatility-stabilized model of \cite{FK05}, \cite{Pi}. Finally, Appendix \ref{section:subsol2} presents   an alternative proof for the viscosity characterizations of the functions $\widehat{\Phi}$ and ${\Phi}$.

\section{Notation and Terminology}\label{section:notation}

We shall fix the dimension $n$, let 
$\Omega:= C([0,\infty); \mathbb{R}_+^n)$ be the canonical space of continuous paths  \label{p:omega}$\omega: [0,\infty) \to \mathbb{R}_+^n$ equipped with the topology of locally uniform convergence.
We shall also 
denote  by \label{p:calF}$\mathcal{F}$ the Borel $\sigma$-field of $\Omega$\,, 
and  \label{p:bbF}$\mathbb{F} = \{\mathcal{F}(t)\}_{0\le t<\infty}$ the raw filtration generated by the canonical process \label{p:B}  $\mathfrak{B}(t,\omega):=\omega(t)$.

We shall let $\, \mathbf{ 0} = (0, \cdots, 0)'$ denote the origin in $\mathbb{R}^n$, and 
\begin{equation}\label{eq:K}
\mathbb{K}=\{\mathcal{K}(y)\}_{y\in [0,\infty)^n\backslash \{\mathbf{ 0}\}}
\end{equation} 
be a collection of nonempty, compact and convex subsets on $\mathbb{R}^n \times \mathbb{S}_{+}(n)$  (recall that $\mathbb{S}_{+}(n)$ is the space of real, symmetric, positive-definite  $n \times n$ matrices).  We denote by \label{p:frakK}$\,\mathfrak{K}\,$ the collection of pairs $(\sigma, \vartheta)$ consisting  of progressively measurable functionals $\sigma = (\sigma_{ik})_{n\times n} : [0,\infty)\times \Omega \to {\mathrm{GL}}(n)$\label{p:sigma} and $\vartheta = (\vartheta_1,\dots,\vartheta_n)^\prime : [0,\infty)\times \Omega \to \mathbb{R}^n$, such that   
\begin{equation}
\label{eq:integrability}
\big(\vartheta(T,\omega),\alpha(T,\omega)\big) \in \mathcal{K}(\omega(T))
\qquad \text{and} \qquad 
\int_0^T \left(||\vartheta(t,\omega)||^2 +\mathrm{Tr}(\alpha(t,\omega))\right)\mathrm{d}t <\infty
\end{equation} 
hold for all $\omega \in \Omega$,   $T \in (0,\infty)$, where 
\begin{equation}\label{eq:alpha}
\alpha := \sigma \sigma^\prime \,.
\end{equation}
Here and throughout the paper, $\, \prime \,$\label{p:prime} denotes transposition and \label{p:GL(n)} ${\mathrm{GL}}(n)$  the space of $n\times n$ invertible real matrices.

\begin{definition}
{\bf Admissible Systems \cite[Sections 1 and 2]{FK11}:} \label{def:M}
{\rm 
For a given  $x = (x_1, \dots, x_n)'\in \mathbb{R}_{+}^n\,$, we
shall call {\em{admissible system}}, subject to the    {\em{Knightian uncertainty}} $\,\mathbb{K}$ with initial configuration $x$, a quintuple $\mathcal{M}= (\sigma,\vartheta, \mathbb{P }, W, \mathfrak{X})$ consisting of 

\smallskip
\noindent
{\bf (i)} a pair $(\sigma,\vartheta)\in \mathfrak{K}$\,; of\\
{\bf (ii)} a probability measure $\mathbb{P}$ on the measurable space $(\Omega,\mathcal{F} )$; of \\
{\bf (iii)} an $n$-dimensional $\mathbb{F}-$Brownian motion $W(\cdot) = (W_1(\cdot), \dots, W_n(\cdot))'$ on the filtered probability space $(\Omega, \mathcal{F}, \mathbb{P}), \mathbb{F}$\,; and of \\
{\bf (iv)} a continuous, $\mathbb{F}-$adapted 
process $\,\mathfrak{X}(\cdot) = (X_1(\cdot), \dots, X_n(\cdot))'$  with values in $\mathbb{R}_{+}^n\,$ and 
\begin{equation}\label{eq:SDE}
\mathrm{d}X_i(t) = X_i(t)\sum_k \sigma_{ik}(t,\mathfrak{X}) \big(\vartheta_k(t,\mathfrak{X})\, \mathrm{d}t + \mathrm{d}W_k(t) \big)\,,\ \  \ \ i = 1,\dots,n\,,\ \ \ \ \mathfrak{X}(0)=x\,.
\end{equation} 
The integrability condition \eqref{eq:integrability} guarantees that the process $\mathfrak{X}(\cdot)$ indeed takes  values in $\mathbb{R}_{+}^n\,$, $\mathbb{P}-$a.s. 

\smallskip
We shall  write $\sigma^{\mathcal{M}}, \vartheta^{\mathcal{M}}, \mathbb{P}^{\mathcal{M}}, W^{\mathcal{M}}$\label{p:^M} and $\mathfrak{X}^{\mathcal{M}}$ for the elements $\sigma,\vartheta, \mathbb{P },  W$ and $\mathfrak{X}$ of the quintuple  $\mathcal{M}$, respectively, and  \label{p:M(x)}$\mathfrak{M}(x)$ for the collection of admissible systems with initial configuration $x \in \mathbb{R}_{+}^n\,$. \qed
}
\end{definition}

In Definition \ref{def:M} and throughout this paper, all vectors are assumed to be column vectors, and summations to extend from $1$ to $n$\,.

\begin{definition}\label{def:MM}
{\bf Strongly Markovian Admissible Systems:} 
{\rm
For a given initial configuration $x \in \mathbb{R}_+^n\,$,  we
shall call {\em{strongly Markovian admissible system}}, subject to the  Knightian  uncertainty $\,\mathbb{K}$ with initial configuration $x$, an admissible system $\mathcal{M}=(\sigma,\vartheta, \mathbb{P}, W, \mathfrak{X})\in \mathfrak{M}(x)$  satisfying:

\noindent
{\bf (i)}
the functionals $\sigma$ and $\vartheta$ are {\it Markovian} and {\it time-homogeneous}, i.e.,
\begin{equation}\label{eq:Markovian}
\sigma(t,\omega)=\mathbf{s}(\omega(t))=(\mathbf{s}_{ij}(\omega(t)))_{1\le i,j\le n}\quad \mathrm{and}\quad \vartheta(t,\omega)=\boldsymbol{\theta}(\omega(t))=(\boldsymbol{\theta}_1(\omega(t)),\dots,\boldsymbol{\theta}_n(\omega(t)))'
\end{equation}
for some measurable functions $\mathbf{s}: \mathbb{R}_+^n\to {\mathrm{GL}}(n)$ and  $\boldsymbol{\theta}: \mathbb{R}^n_+\to \mathbb{R}^n$; and

\noindent
{\bf (ii)}
for every $y\in \mathbb{R}_{+}^n$\,, there exists an admissible system $\mathcal{M}^y\in \mathfrak{M}(y)$ with the same $\mathbf{s}(\cdot)$ and $\boldsymbol{\theta}(\cdot)$ as in ${\mathcal{M}}$, and a strongly Markovian state process $\mathfrak{X}(\cdot)$.

We shall denote by  $\,{\widehat{\mathfrak{M}}}(x)\,$ the subcollection of $\mathfrak{M}(x)$ consisting of all strongly Markovian admissible systems with initial configuration $x$.
}\qed
\end{definition}

\begin{remark}
{\rm
It follows from the Markovian selection results of \textsc{Krylov} (see \cite{K73}, \cite[Chapter 12]{SV} and \cite[Theorem 5.4]{EK}) that, 
if the collection of subsets $\mathbb{K}$ satisfies the linear growth condition   
\begin{equation} \label{eq:KLinearGrow}
\sup_{(\theta,a)\in \mathcal{K}(y),\,b=\varsigma\theta,\, \varsigma\varsigma'=a} 
\left[\sum_{i,j}y_i y_j a_{ij} +\sum_i (y_i b_i)^2\right]
\le C(1+||y||)^2, \ \  \forall ~ y\in [0,\infty)^n \backslash \{\mathbf{ 0}\}\,,
\end{equation}
for some constant $C>0$\,, then the state process $\mathfrak{X}(\cdot)$ can be chosen to be
strongly Markovian under $\mathbb{P}^{\mathcal{M}}$ for any admissible system $\mathcal{M}$ with Markovian and time-homogeneous $\sigma$ and $\vartheta$ as  in \eqref{eq:Markovian}.
}  \qed
\end{remark}

\begin{remark}\label{rmk:nonemptyM}
{\rm 
{\bf (i)} A sufficient condition for  
$\, \widehat{\mathfrak{M}} (x)\neq \emptyset \,$ to hold for all $x\in \mathbb{R}_+^n$\,, is that there exist locally \textsc{Lipschitz} functions $\mathbf{s}(\cdot)$ and  $\boldsymbol{\theta}(\cdot)$ satisfying Condition (i) of Definition \ref{def:MM}, that 
$\big(\boldsymbol{\theta}(y),\mathbf{s}(y)\mathbf{s}'(y)\big)$ $\in \mathcal{K}(y) $  for all  $y\in \mathbb{R}^n_+$\,,
and that $\mathbf{s}(\cdot)$ and  $\mathbf{b}(\cdot):=\mathbf{s}(\cdot)\boldsymbol{\theta}(\cdot)$ are  linearly growing, i.e., \begin{equation} \label{eq:linearGrowth}
||\mathbf{s}(y)||+||\mathbf{b}(y)||\le C(1+ ||y||)\quad {\mathrm{for\ all\ }}  y\in \mathbb{R}^n_+\,, 
\end{equation}
for  some real constant $\, C>0\,$. Under this condition and for any   $x\in \mathbb{R}^n_+$\,,  the SDE \eqref{eq:SDE} with the $\sigma$ and $\vartheta$ as in \eqref{eq:Markovian}, always has a pathwise unique, strong solution starting at $x$   (\cite[Theorem 5.2.2]{FR}; \cite[p.\,8]{T}).

\smallskip
{\bf (ii)} In particular, if 
$\mathcal{K}(y)=\left\{\big(\boldsymbol{\theta}(y),\mathbf{s}(y)\mathbf{s}'(y)\big)\right\}$  for all  $y\in \mathbb{R}^n_+$ with such  $\mathbf{s}$ and  $\boldsymbol{\theta}$, then we have $\, {\mathfrak{M}}(x) = \widehat{\mathfrak{M}}(x)\neq \emptyset\,$ for all $x\in \mathbb{R}_+^n$\,.
} \qed
\end{remark} 

\section{Interpretation and Previous Results}\label{section:previous}

The above variables can be interpreted in a model for an equity market with $n$ assets, say stocks, as follows: 

\smallskip
\noindent
\textbf{(i)} $\mathfrak{X}(t)$ as the vector of capitalizations for the various assets $i=1, \cdots, n$ at time $t\,$, and  
\begin{equation}\label{eq:X(t)}
X(t)\,:=\,\sum_i X_i(t)
\end{equation}
as the total capitalization at that time; \\ \textbf{(ii)} $W(\cdot)$ as the vector of independent factors (sources of randomness)  in the resulting model; \\ 
\textbf{(iii)} $\sigma_{ik}(t,\mathfrak{X}), \, k=1, \cdots, n\,$ as the local volatilities for the $i$th asset at time $\,t\,$; \\ 
\textbf{(iv)} $\alpha_{ij}(t,\mathfrak{X})$\label{p:alpha_{ij}} as the local covariation rate between assets $i$ and $j$ at time $\,t\,$; \\ 
\textbf{(v)} $~\vartheta(t,\mathfrak{X})$ as the vector of local market prices of risk at time $\,t\,$; and \\ 
\textbf{(vi)} \label{p:beta}$\beta(t,\mathfrak{X}):=(\sigma\vartheta)(t,\mathfrak{X})$ as the vector of local rates of return at time $t\,$\,.

\subsection{Investment Rules and Portfolios}

Consider now an investor who is ``small", in the sense
that his actions have no effect on market prices. Starting with initial fortune $v > 0$\label{p:v}, he uses a rule that invests a proportion ${\it \Pi}_i (t,\mathfrak{X})$\label{p:Pi_i} of current wealth in the $i$-th asset of the equity market at time $t \in [0,\infty)$ $(i = 1, \dots, n)$, and holds the remaining proportion in cash -- or equivalently  in a zero-interest money market. 

\smallskip
We shall call {\it investment rule}  a progressively measurable functional  \label{p:P}${\it \Pi}=({\it \Pi}_1,\,\cdots,{\it \Pi}_n)':  [0,\infty)\times\Omega \to \mathbb{R}^n$\label{p:Pi}   satisfying
\begin{equation}\label{eq:Pi} 
\int_0^T  \big(\left|{\it \Pi}^\prime(t,\omega) \sigma  (t,\omega)\vartheta(t,\omega)\right| + {\it \Pi}^\prime(t,\omega) \alpha (t,\omega) {\it \Pi}(t,\omega)\big)\,\mathrm{d}t <\infty\ \ \quad {\mathrm{for\ all\ }} \,\,T \in(0,\infty)\,,\ \omega\in \Omega\,,
\end{equation}
and  denote by $ \mathfrak{P}$  the set of all such (nonanticipative) investment rules. 

\smallskip
We shall call an investment rule ${\it \Pi}$ {\it bounded}\label{p:bdd}, if ${\it \Pi}$ is bounded uniformly on $[0,\infty)\times \Omega$\,; for a bounded investment rule, the requirement \eqref{eq:Pi} is satisfied automatically,
on the strength of \eqref{eq:integrability}. 

\smallskip
We shall call an investment rule ${\it \Pi}$ a {\it portfolio,}\label{p:portfolio} if $\sum_i {\it \Pi}_i=1$ on $[0,\infty)\times \Omega$\,; in other words, if it never invests, in or borrows
from, the money market. 
We shall call a portfolio ${\it \Pi}$ {\it long-only}\label{p:long} if ${\it \Pi}_i\ge 0$\,, $i=1,\dots,n$ also holds on this domain, that is, it never sells any stock short. A long-only portfolio is also bounded, since it satisfies $0\le {\it \Pi}_i\le 1$\,, $i=1,2,\dots,n$.

\medskip
\noindent
$\bullet~$ 
Given an initial wealth $v$,  an investment rule ${\it \Pi}$ and an admissible
model $\mathcal{M}\in \mathfrak{M}(x)$, the resulting wealth process $Z(\cdot):=Z^{\,v,{\it \Pi}}(\cdot)$
satisfies the initial condition 
$Z(0)=v $ and 
\begin{equation}\label{eq:Z}
\frac{\mathrm{d}Z(t)}{Z(t)}
= \sum_{i}  {\it \Pi}_i(t,\mathfrak{X}) \frac{\mathrm{d}X_i(t)}{X_i(t)}
= {\it \Pi}'(t,\mathfrak{X})\sigma(t,\mathfrak{X}) \left[\vartheta(t,\mathfrak{X})\, \mathrm{d} t+ \mathrm{d}W(t) \right] \,,\quad \mathrm{by}\ \eqref{eq:SDE}\,.
\end{equation}

\subsection{The Market Portfolio}

In the special case with 
\begin{equation}\label{eq:mu}
{\it \Pi}_i(t,\omega)\,\equiv \, {\mu}_i(t,\omega)\,:=\, \frac{\omega_i(t)}{ \,\omega_1(t) + \cdots + \omega_n(t)\,} \,, \qquad \forall \,\,\,\, i=1,  \cdots, n\,, \,\,\,0 \le t<\infty\,,
\end{equation}
we have  ${\mu}_i(\cdot\,,\mathfrak{X})=X_i(\cdot)/X(\cdot)$: the resulting   strategy $\, \mu\,$ invests in all stocks    in proportion to their relative market weights. We call the resulting strategy ${\it \Pi} \equiv \mu $ the  (long-only) {\it market portfolio\label{p:mu}.} It follows from the first equality in the dynamics \eqref{eq:Z} that investing according to the market portfolio 
amounts to owning the entire market, in proportion of course to the initial wealth:  $Z^{\,v,\mu}(\cdot)=vX(\cdot)/X(0)$.  

\subsection{The Arbitrage Function}

With these ingredients in place, we define the {\em{arbitrage function}}\label{p:arbitrageFn}
$\mathfrak{u}: [0,\infty)\times \mathbb{R}_{+}^n \to (0,1]$,  as
\begin{eqnarray}\label{eq:u}
&\ \ \  &  \mathfrak{u}(T, x) := \inf \left\{r >0 : \exists\, {\it \Pi} \in \mathfrak{P}\  \text{s.t.}\  \mathbb{P}^{\mathcal{M}} \left[Z^{\,rX^{\mathcal{M}}(0),{\it \Pi}}(T) \ge X^{\mathcal{M}}(T) \right]= 1\,, \,\,\forall\, \,\mathcal{M}\in \mathfrak{M}(x)\right\}.
\end{eqnarray}
For the strict positivity of this quantity, see \eqref{eq:u>=Phi} below.   \qed

\smallskip
We call the function $\mathfrak{u}(\cdot,\cdot)$ the arbitrage function  because, for the initial configuration    $x=(x_1,\dots,x_n)'\in \mathbb{R}_+^n$ of asset capitalizations, the quantity $\mathfrak{u}(T, x)$ can be  thought of   as the smallest proportion of the initial total market capitalization $x_1+\cdots+x_n$\,, starting with which one can find a nonanticipative investment rule, whose performance matches or outperforms that of the market portfolio over the time horizon $[0,T]$,      with probability one under all admissible systems. Equivalently,   $\mathfrak{u}(T, x)$ can be  thought of    as the reciprocal of the highest return on investment relative to the market portfolio over the time horizon $[0,T]$, that can be  achieved using nonanticipative investment rules when starting with the vector $x=(x_1,\dots,x_n)' \in \mathbb{R}_{+}^n$ of initial capitalizations, and with probability one under all admissible systems. 

\smallskip
Given an admissible system $\mathcal{M}  =(\sigma,\vartheta, \mathbb{P }, W, \mathfrak{X}) \in \mathfrak{M}(x)$, we define the  stochastic discount factor $L(\cdot)$ as the associated exponential $\,\mathbb{P}-$local martingale 
\begin{equation}\label{eq:L=exp}
L(t)\, :=\, \exp\left(-\int_0^t\vartheta^\prime (s,\mathfrak{X})\, \mathrm{d}W(s) - \int_0^t \frac{1}{\,2\,} \,\big|\big|\vartheta(s,\mathfrak{X})\big|\big|^2 \mathrm{d}s\right), \quad  0\leq t < \infty\,.
\end{equation}
This process is well-defined and a strictly positive $\,\mathbb{P}-$local martingale (thus a $\,\mathbb{P}-$supermartingale), on the strength of the integrability condition \eqref{eq:integrability}; {\it but is not necessarily a $\,\mathbb{P}-$martingale}. It plays the r\^{o}le of a state-price-density or ``deflator" in the present context. We also write $L^{\mathcal{M}}(\cdot)$\label{p:L^M} for this $L(\cdot)$ under $\mathcal{M}$ when needed.

\smallskip
Assuming that $\, \widehat{\mathfrak{M}}(x)\neq \emptyset\,$ holds for all $\, x\in \mathbb{R}_+^n\,$, we consider the functions
\begin{equation}\label{eq:Phi}
\Phi(T, x) := \sup_{\mathcal{M}\in \mathfrak{M}(x)} \mathfrak{u}_{\mathcal{M}}(T, x) \qquad  {\mathrm{and}}  \qquad 
\widehat{\Phi}(T, x) := \sup_{\mathcal{M}\in \widehat{\mathfrak{M}}(x)} \mathfrak{u}_{\mathcal{M}}(T, x) 
\end{equation}
for $\, (T, x)\in [0,\infty)\times \mathbb{R}_{+}^n\,$, where 
\begin{equation}\label{eq:um}
\mathfrak{u}_{\mathcal{M}}(T, x) \,:= \,\mathbb{E}^{\mathbb{P}^{\mathcal{M}}} \left[L^{\mathcal{M}}(T)X^{\mathcal{M}}(T)\right] /\, ||x||_1
\end{equation}
(recall the $\,\ell^1$-norm $\,||x||_1=\sum_i x_i\,$) and the total capitalization \begin{equation}\label{eq:X^M(T)}
X^{\mathcal{M}}(T):=||\mathfrak{X}^{\mathcal{M}}(T)||_1=\sum_i X^{\mathcal{M}}_i(T)\,.
\end{equation}

As was shown in \cite[Section 10, pp.\,127--129]{FK09}, \cite{KS}, or \cite{R11}, the quantity  $\mathfrak{u}_{\mathcal{M}}(T, x)$ in \eqref{eq:um} is obtained by fixing an admissible system $\mathcal{M}$ in the  definition \eqref{eq:u} of $\mathfrak{u}$\,, namely,
\begin{equation}\label{eq:u_M}
\mathfrak{u}_{\mathcal{M}}(T, x) \,= \,\inf \left\{r >0 : \exists\ {\it \Pi} \in \mathfrak{P}\ \ \text{s.t.}\ \ \mathbb{P}^{\mathcal{M}} \left[Z^{\,rX^{\mathcal{M}}(0),{\it \Pi}}(T) \ge X^{\mathcal{M}}(T) \right]= 1\right\} \in ( 0,1]\,.
\end{equation}
This can be interpreted as the reciprocal of the highest return on investment over the time horizon $[0,T]$, that can be achieved relative to the market portfolio in the context of the model $\mathcal{M}$, by using nonanticipative strategies and starting with the vector $x$ of initial capitalizations. It can also can be interpreted as the {\em{arbitrage function for}} $\mathcal{M}$\label{p:u_M} in the terminology of \cite[Section 6]{FK10}, at least when $(\mathbb{P}, \mathbb{F})$-martingales can be represented as stochastic integrals with respect to the $W(\cdot)$  in \eqref{eq:SDE}.   

Since the processes $L(\cdot)$ and $X(\cdot)$ are strictly positive, so is the function $\mathfrak{u}_{\mathcal{M}}(\cdot\,,\cdot)$ for all admissible system $\mathcal{M}$. 
It then follows from the definitions \eqref{eq:u}--\eqref{eq:u_M}  that 
\begin{equation}\label{eq:u>=Phi}
1\ge \mathfrak{u}(T, x)\ge \Phi(T,x) \ge \widehat{\Phi}(T,x)>0\,,\ \ \ \ \  \forall \ (T, x)\in [0,\infty)\times \mathbb{R}_{+}^n\,.
\end{equation}

\begin{remark}\label{rmk:u<1}
{\rm 
{\bf Strong Arbitrage:}
If $ \mathfrak{u} (T,x) <1 $, then a
{\it strong arbitrage} relative to the market portfolio in the terminology of \cite[Definition 6.1]{FK09} exists on $[0,T]$ with the initial capitalizations $x$. Such strong arbitrage  is {\it robust}\label{p:robust}, that is, holds under every possible admissible system or model that might materialize.

Instances of $ \mathfrak{u} (T,x) <1 $ with $ T\in (0,\infty)$ occur when there exists a real constant $C>0 $ such that either
\begin{equation*}
\inf_{a\in \mathcal{A}(y)} \left( \sum_{i}  \frac{y_i a_{ii}}{y_1+ \cdots+ y_n}  - \sum_{i,j}  \frac{y_i y_j a_{ij}}{(y_1 + \cdots+ y_n )^2 } \right) \ge C
\end{equation*}
or
\begin{equation*}
\frac{(y_1 \cdots y_n )^{1/n}}{y_1 + \cdots+ y_n}   \cdot \inf_{a\in \mathcal{A}(y)}
\left(  \sum_{i} a_{ii}  - \frac{1}{\,n\,}  \sum_{i,j}
a_{ij}\right) \ge C
\end{equation*}
holds for every $y \in\mathbb{R}_+^n $ (recall $\mathcal{A}(\cdot)$ from \eqref{eq:A} and see   \cite[Examples 11.1, 11.2]{FK09}, \cite{FK05} and \cite{FKK}).
}  \qed
\end{remark}

\begin{remark}\label{rmk:NUPBR}
{\rm
{\bf No Unbounded Profits with Bounded Risk:}
The inequality $\mathfrak{u}(T, x)>0$ in  \eqref{eq:u>=Phi} rules out scalable arbitrage opportunities, also known as 
{\it Unbounded Profits with Bounded Risk (UPBR)}. We refer the reader to \cite{DSch} for the origin of the resulting ``No Unbounded Profit  with Bounded Risk" (NUPBR)  concept, and to \cite{KK} for an elaboration of this point in a different context, namely, the existence and properties of the so-called ``num\'{e}raire" portfolio.
}  \qed
\end{remark}

\subsection{Previous Results}

The  Knightian  uncertainty in the above model 
shares a lot  with the uncertainty regarding the underlying volatility structure of assets in \cite{Lyons}.
The approach in \cite{FK11} is reminiscent of the \textsc{Dubins-Savage} (\cite{DS}) and \textsc{Sudderth} (\cite{HOPS}, \cite{OPS}, \cite{PS}, \cite{SW}) approaches to stochastic optimization. 

The arbitrage function $\,\mathfrak{u}\,$ of \eqref{eq:u} was characterized in \cite{FK11}     as a classical solution and in fact, the smallest nonnegative classical (super)solution of the \textsc{Cauchy} problem \eqref{eq:PDE}, \eqref{eq:initial cond},  but under rather strong assumptions on the uncertainty structure (see Theorem \ref{thm:Knightian} below), which amount to: 
$\Phi\equiv \mathfrak{u}_{\mathcal{M}}$ for some strongly Markovian admissible system ${\mathcal{M}}$, and
$\mathfrak{u}_{\mathcal{M}}$ solves \eqref{eq:PDE}.

\begin{theorem}\cite[Proposition 3, Remark 2]{FK11}\label{thm:Knightian}
We have 
$\mathfrak{u}\equiv\Phi$ on $[0, \infty)\times \mathbb{R}_{+}^n$\,, and 
in fact, this function is the smallest nonnegative  (super)solution of \eqref{eq:PDE}, \eqref{eq:initial cond}, if there exists a strongly Markovian admissible system $\mathcal{M}_o$ under which EITHER:

\smallskip
\noindent
{\bf (i)}
the functions $\mathbf{s}$ and $\boldsymbol{\theta}$ of \eqref{eq:Markovian} are locally L{\scriptsize IPSCHITZ}, and

\noindent
{\bf (ii)}
the function $u(t,x):=\mathfrak{u}_{\mathcal{M}^x_o}(t,x)$, which, by \cite[Theorem 4.7]{R13}, is of class $C^{1,2}$ and solves 
\begin{equation}\label{eq:a}
\left( u_t-\mathcal{L}_{\mathbf{a}(x)}u\right) (t,x)=0\quad \mathrm{with}\quad  \mathbf{a}(x):=\mathbf{s}(x)\mathbf{s}'(x)\,, \quad (t,x)\in (0, \infty)\times\mathbb{R}_{+}^n\quad 
\end{equation} 
($\mathcal{M}^x_o:=(\mathcal{M}_o)^x$; recall Definition \ref{def:MM} for $\mathcal{M}^y$ and \eqref{eq:hatL} for $\mathcal{L}_{a}$), 
is a classical supersolution of \eqref{eq:PDE}; 

\medskip
OR  both of the following conditions hold:

\smallskip
\noindent
{\bf (i)$^\prime$} the functions $\mathbf{s}$ and $\boldsymbol{\theta}$ of  \eqref{eq:Markovian} are continuous,

\noindent
{\bf (ii)$^\prime$} 
there exists a positive constant $C$ such that\, $$\sum_{i,k}y_i|\mathbf{s}_{ik}(y)\boldsymbol{\theta}_k(y)| \,\le \,C(1+||y||)$$ holds for all $y\in \mathbb{R}_{+}^n\,,$

\noindent
{\bf (iii)$^\prime$} 
there exists a $C^2$-function $\mathbf{h}:\mathbb{R}_{+}^n \to \mathbb{R}$ such that $\,\boldsymbol{\theta}_k(y)=\sum_i y_i \mathbf{s}_{ik}(y)D_i\mathbf{h}(y)$,  $\,k=1,\dots,n\,,$  

\noindent
{\bf (iv)$^\prime$} the function 
$$\mathscr{G}(t,x)\,:=\, \mathbb{E}^{\mathbb{P}^{\mathcal{M}^x_o}}\left[\mathscr{F}(\mathfrak{X}(t))\exp\left( \int_0^t \mathscr{K} \big(\mathfrak{X}(t) \big) \right)  \right] \in C\left( [0,\infty)\times \mathbb{R}_{+}^n\right) \cap C^{1,2}\left( (0,\infty)\times \mathbb{R}_{+}^n\right), $$
where
$$\mathscr{F}(y)\,:=\,\frac{1}{\,2\,}\sum_{i,j} \mathbf{a}_{ij}(y)\left[ D^2_{ij}\mathbf{h}+D_i\mathbf{h}\cdot D_j\mathbf{h}\right](y) \qquad \text{ and } \qquad \mathscr{K}(y)\,:=\,||y||_1\exp \big(-\mathbf{h}(y)\big)\, , $$
and

\smallskip
\noindent
{\bf (v)$^\prime$} the function $ \mathscr{U}(t,x)\,:=\,\mathscr{G}(t,x)\,/\mathscr{F}(x)$ is a classical supersolution of \eqref{eq:PDE}.
\end{theorem}

A natural question to ask then, is whether the arbitrage function $\mathfrak{u}$ of \eqref{eq:u} is still a solution to \eqref{eq:PDE},   
perhaps in some weak or generalized sense, when regularity and other conditions are weakened. The answer turns out to be affirmative, though it is somewhat indirect; it is provided in Theorems \ref{thm:viscosity1}, \ref{thm:viscosity2} and Corollary \ref{coro:u sol} below.

\section{Viscosity Characterizations of the functions ${\Phi}$ and $\widehat{\Phi}$}
\label{section:visc}

We first recall from \cite{CIL} the definition of viscosity (sub/super)solutions  for a second-order parabolic partial differential equation, and then state our main results with a  discussion of related results. 

\subsection{Viscosity (Super/sub)solution of a Second-order Parabolic PDE}

Let $\mathcal{O}$\label{p:O} be an open subset of $\mathbb{R}^n$, let $\mathbb{S}(n)$\label{p:S(n)} be the set of $n\times n$ real symmetric matrices,  and consider a continuous, real-valued mapping $(t,x,r,p,q) \mapsto F(t,x,r,p,q)$\label{p:F}  defined on $(0,\infty)\times \mathcal{O}\times \mathbb{R}\times \mathbb{R}^n\times \mathbb{S}(n)$ and satisfying the {\it ellipticity condition}
\begin{equation}\label{eq:ellipticity}
\ F(t,x,r,p,q_1) \leq F(t,x,r,p,q_2) {\mathrm{ \ \ whenever\ }} \ q_1 \geq q_2\,,\quad {\mathrm{for\ all\ }}(t,x,r,p) \in (0,\infty)\times \mathcal{O}\times \mathbb{R}\times \mathbb{R}^n.
\end{equation}

Consider the second-order parabolic partial differential equation
\begin{equation}\label{eq:F}
u_t + F\left(t,x,u(t,x),Du(t,x),D^2 u(t,x)\right) = 0\,,\quad (t,x)\in (0,\infty)\times \mathcal{O}
\end{equation}
with the gradient $Du=(u_{x_1}, u_{x_2}, \dots, u_{x_n})'$ and the Hessian $D^2 u=(u_{x_i x_j})_{n\times n}$\label{p:Du}\,.

\begin{definition}\label{def:viscosity sol}
{\bf Viscosity Solution:}
{\bf (i)} We say that a function $\,u :(0,\infty) \times \mathcal{O}\to \mathbb{R}\,$ is a {\bf viscosity subsolution}  of the equation \eqref{eq:F}, if
\begin{equation}\label{eq:F<=0}
\varphi_t + F\left(t_0,x_0,u^*(t_0,x_0),D\varphi(t_0,x_0),D^2 \varphi(t_0,x_0)\right)\leq 0
\end{equation}
holds for all $(t_0,x_0)\in (0,\infty) \times \mathcal{O}$ and test functions $\varphi \in C^{1,2}\left( (0,\infty) \times \mathcal{O}\right) $ such that $(t_0,x_0)$ is a (strict) (local) maximum of $\,u^*-\varphi\,$ on $(0,\infty) \times \mathcal{O}$. We have denoted here by 
\begin{equation}\label{eq:u^*}
u^*(t,x) := \limsup_{(s,y)\rightarrow (t,x)} u(s,y),\ \ \ \ (t,x)\in (0,\infty)\times \mathcal{O}
\end{equation}
the {\em upper-semicontinuous envelope} of $u$, i.e., the smallest upper-semicontinuous function that dominates pointwise the function $u$.

{\bf (ii)} Similarly, we  say that  $\,u :(0,\infty) \times \mathcal{O}\to \mathbb{R}\,$ is a {\bf viscosity supersolution}  of  \eqref{eq:F}, if
\begin{equation}\label{eq:F>=0}
\varphi_t + F\left(t_0,x_0,u_*(t_0,x_0),D\varphi(t_0,x_0),D^2 \varphi(t_0,x_0)\right)\geq 0
\end{equation}
holds for all $(t_0,x_0)\in (0,\infty) \times \mathcal{O}$ and test functions $\varphi  \in C^{1,2}\left( (0,\infty) \times \mathcal{O}\right) $ such that $(t_0,x_0)$ is a (strict) (local) minimum of $\,u_*-\varphi\,$  on $(0,\infty) \times \mathcal{O}$. We have denoted here by 
\begin{equation}\label{eq:u_*}
u_*(t,x) : =  \liminf_{(s,y)\rightarrow (t,x)} u(s,y),\ \ \ \ (t,x)\in (0,\infty)\times \mathcal{O}
\end{equation}   
the {\em lower-semicontinuous envelope} of $u$, i.e., the largest lower-semicontinuous function dominated pointwise by the function $u$.

{\bf (iii)} Finally, we say that $\,u :(0,\infty) \times \mathcal{O}\to \mathbb{R}\,$ is a {\bf viscosity solution}  of \eqref{eq:F}, if it is both a viscosity subsolution and a viscosity supersolution of this equation.
\end{definition}

\begin{remark}\label{rmk:u*}
{\rm 
The above definition implies that  $u$ is a viscosity subsolution (supersolution) of \eqref{eq:F} if and only if $u^*$ ($u_*$) is a viscosity subsolution (supersolution)  of this equation.
} \qed
\end{remark}

\subsection{Main Results}

In our setting we have $\mathcal{O}\,=\,\mathbb{R}_{+}^n$ and 
\begin{equation} \label{eq:ourF}
F(t,x,r,p,q) \,=\, - \sup_{a\in \mathcal{A}(x)} \left( \sum_{i,j}  x_i x_j a_{ij} \left(\frac{q_{ij}}{2} + \frac{p_i}{||x||_1}\right)\right)\    
\end{equation}  
for $\,q=(q_{ij})_{1\le i,j\le n}\,,\ p=(p_1,\dots,p_n)',$ and thus the left-hand sides of \eqref{eq:F<=0} and \eqref{eq:F>=0} simplify to $\big(\varphi_t-\widehat{\mathcal{L}}\varphi \big)(t_0,x_0)$ in the notation of \eqref{eq:hatL}. 

Since each matrix  $a$ in the collection $ \mathcal{A}(x)$ of \eqref{eq:A} is positive-definite, we deduce that the matrix $(x_i x_j a_{ij})_{n\times n}$ $ = x^\prime a x$ is always positive-definite, and hence $F$ satisfies the ellipticity condition \eqref{eq:ellipticity}. 

In the results that follows, we shall also need $F$ to be a continuous mapping, as well as the following conditions:

\begin{asmp} {\bf Local Boundedness:} 
\label{asmp1}
 The collection $\mathbb{K}$ of \eqref{eq:K}
is {\rm locally bounded}  on $\mathbb{R}_{+}^n\,$; that is, for any $x\in \mathbb{R}_{+}^n$\,, there exists a neighborhood $\mathcal{D}(x)\subset \mathbb{R}_{+}^n$\label{p:D(x)} of $x$ such that $\,\bigcup_{y\in \mathcal{D}(x)}\mathcal{K}(y)$ is bounded. 
\end{asmp}

\begin{asmp} {\bf  Continuity:} 
\label{asmp2}{\rm 
For any $\iota>0$\,, $x\in \mathbb{R}_{+}^n$ and $a=\left(a_{ij} \right)_{1\le i,j\le n}\in \mathcal{A}(x)$, there exist a positive number $\delta<\iota$ and locally \textsc{Lipschitz} functions $\mathbf{s}: \mathbb{R}_+^n\to {\mathrm{GL}}(n)$ and  $\boldsymbol{\theta}: \mathbb{R}^n_+\to \mathbb{R}^n$ such that $\mathbf{s}(\cdot)$ and  $\mathbf{b}(\cdot):=\mathbf{s}(\cdot)\boldsymbol{\theta}(\cdot)$ are  linearly growing (i.e., satisfy the condition  \eqref{eq:linearGrowth}), and that for $\mathbf{a}(\cdot)=\left(\mathbf{a}_{ij}(\cdot) \right)_{1\le i,j\le n} :=\mathbf{s}(\cdot)\mathbf{s}'(\cdot)$, we have
$\big(\boldsymbol{\theta}(y),\mathbf{a}(y)\big)$ $\in \mathcal{K}(y) $  for all  $y\in \mathbb{R}^n_+$ and 
\begin{equation} \label{eq:continuous a}
|\mathbf{a}_{ij}(y) - a_{ij}|<\iota\,,\quad  1\le i,j\le n\,, \quad {\mathrm{for\ all}}\  y\in B_{\delta}(x)\,.
\end{equation} 
}
\end{asmp}

\noindent
{\it Remark:}
All of the conditions in Assumption \ref{asmp2}, except for \eqref{eq:continuous a}, are inspired by Remark \ref{rmk:nonemptyM}. The aim  is  to guarantee the existence of an admissible system with the functional $\sigma(t,\omega)=\mathbf{s}(\omega(t))$, as in \eqref{eq:Markovian}. \qed

\smallskip
We have then the following results.
 
\begin{theorem}\label{thm:viscosity1}
{\bf Viscosity Subsolution:} Suppose that the real-valued function $F$ of \eqref{eq:ourF} is continuous on $(0,\infty) \times \mathbb{R}_+^n\times \mathbb{R} \times \mathbb{R}^n\times \mathbb{S}(n)$, and that Assumption \ref{asmp1} holds.   

The function $\widehat{\Phi}$ of \eqref{eq:Phi} is then a viscosity subsolution of the HJB equation \eqref{eq:PDE}, and thus a viscosity subsolution of the  C{\scriptsize AUCHY} problem \eqref{eq:PDE}, \eqref{eq:initial cond},  since it satisfies $\widehat{\Phi}(0,\cdot)=1$.
\end{theorem}

\begin{theorem}\label{thm:viscosity2}
{\bf Viscosity Supersolution:}  
Suppose that the real-valued function $F$ of \eqref{eq:ourF} is continuous on $(0,\infty) \times \mathbb{R}_+^n\times \mathbb{R} \times \mathbb{R}^n\times \mathbb{S}(n)$, and that Assumptions \ref{asmp1} and \ref{asmp2} hold.   
 
The function $\Phi$ of \eqref{eq:Phi} is then a viscosity supersolution of the HJB equation \eqref{eq:PDE}, and thus a viscosity supersolution of the C{\scriptsize AUCHY} problem \eqref{eq:PDE}, \eqref{eq:initial cond}, since  it satisfies $\Phi(0,\cdot)=1$.
\end{theorem}

\subsection{Discussion of Related Work}
 
These results echo similar themes from the   literature on   models with an analogous type of uncertainty, under which the functionals $\sigma$ and $\vartheta$ are {\it fixed}; instead, the uncertainty  comes from a {\it control} process $\mathcal{C}(\cdot)$. 
At any time $t$, the values of $\sigma$ and $\vartheta$ are determined not only by the present capitalizations $\mathfrak{X}(t)$, but also by the present value $\mathcal{C}(t)$ of the control process $\mathcal{C}$, i.e., the local volatility matrix and the relative risk vector at time $t$ are $\sigma(\mathfrak{X}(t),\mathcal{C}(t))$ and $\vartheta(\mathfrak{X}(t),\mathcal{C}(t))$, respectively. 
A control process is a progressively measurable process that takes values in a given subset $\Gamma$ of some  Euclidean space and satisfies certain integrability condition. 

\smallskip
Among those papers in the literature are the ground-breaking works  \cite{L83a}--\cite{L84} by \textsc{P.L. Lions}, specifically, 
\cite[Theorem III.1]{L84} (or \cite[Theorem I.1]{L83b}). These impose much stronger assumptions   on the volatility and drift structure: namely, $\sup_{\gamma\in\Gamma} ||h(\cdot,\gamma)||_{W^{2,\infty}(\mathbb{R}^n)}< \infty$, 
and continuity of $h(x,\cdot)$ for all $x$, where \label{p:h}$h=\sigma_{ik},\beta_i$\,, $1\le i,k\le n$\,. 

A similar result was proved in \cite[Theorem 4.1]{ST}, but under the  stronger assumptions that both functions $\sigma$ and $\beta$ be  bounded and \textsc{Lipschitz}, that the analogue  in their formulation of the function $F$ of \eqref{eq:ourF} be locally \textsc{Lipschitz}, and that the set $\Gamma$ be compact. 

\smallskip
If the functions $\alpha_{ij}(\cdot,\gamma)$ $( \gamma\in \Gamma, 1\le i,j\le n)$ are all of class $C^{1,\eta}_{\mathrm{loc}}\left(\mathbb{R}_+^n\right)$ for some constant \label{p:eta}$\eta\in (0,1]$, then in \cite[Theorem 3.3]{BH}, and more generally, in \cite[Theorem 2.1]{KR}, the asymptotic-growth-optimal trading strategy is characterized in terms of a generalized version of the principal eigenvalue of the following fully nonlinear elliptic operator and its associated eigenfunction:
$$\widetilde{\mathcal{L}} u (t,x)\, :=\, \frac{1}{\,2\,}\sum_{i,j}  x_i x_j a_{ij}  D^2_{ij} u(t,x) \,.
$$

For a model with {\it no} uncertainty and with local volatility matrix $\sigma(\mathfrak{X}(t))$ and  relative risk vector $\vartheta(\mathfrak{X}(t))$ at time $t$, the viscosity characterization was obtained in \cite[Proposition 4.5]{BHS} but with additional local \textsc{Lipschitz} condition on $\sigma$ and $\vartheta$. This (local) \textsc{Lipschitz} condition is also a typical assumption in previous literature on stochastic control and dynamic programming, e.g., \cite{BT}, \cite{HL} and \cite{T} (it is even assumed in \cite{FV} that $\sigma(y,\gamma)$ and $\vartheta(y,\gamma)$ are continuous and twice differentiable in $y$). 

\smallskip
In the one-dimensional case $(n=1)$ with {\it zero}  drift $(\beta\equiv 0)$ but  {\it no}  uncertainty, the authors of \cite{CHSZ} removed the local \textsc{Lipschitz} condition and hence chose not to pursue a viscosity characterization; instead, provided that the function $\sigma$ is  continuous and satisfies $\int_{1}^{\infty} x\sigma^{-2}(x)\,\mathrm{d}x =\infty\,$,  
they approximated the arbitrage function by classical solutions to  \textsc{Cauchy} problems \cite[Theorem 5.3]{CHSZ}.

\section{The Proof of Theorem \ref{thm:viscosity1}: Viscosity Subsolution}
\label{section:subsol}

We first highlight the main idea without many of the technicalities. 
We argue by contradiction, assuming the negation of \eqref{eq:F<=0} in Definition \ref{def:viscosity sol} with the function $F$ as in \eqref{eq:ourF}: namely,  that there exist   $\varphi \in C^{1,2}\left( (0,\infty) \times \mathbb{R}_{+}^n\right)$ and  $(t_0,x_0) \in (0,\infty) \times \mathbb{R}_{+}^n$,  such that  $(t_0,x_0)$ is a strict maximum of $ \, \widehat{\Phi}^* - \varphi \,$; that the maximal value is equal to zero; and that 
\begin{equation}\label{eq:hatG>0}
\big(\varphi_t-\widehat{\mathcal{L}}\varphi\big) (t_0,x_0)>0\,. 
\end{equation}

It follows from the definition \eqref{eq:u^*} of $\widehat{\Phi}^*$ that we can take a pair $(t^*,x^*)$ close to $(t_0,x_0)$ such that the nonnegative difference $\big(\varphi-\widehat{\Phi}\big) (t^*,x^*)$ is sufficiently small, say less than a small positive constant $C_3$\,; further, by the definition \eqref{eq:Phi} of $\widehat{\Phi}$, we can take an admissible system $\mathcal{M}^{x^*}\in \widehat{\mathfrak{M}}(x^*)$ such that $0\le \big(\widehat{\Phi}-\mathfrak{u}_{\mathcal{M}^{x^*}} \big)(t^*,x^*)<C_3$\,. Therefore $0\le(\varphi-\mathfrak{u}_{\mathcal{M}^{x^*}} )(t^*,x^*)<2\,C_3$\,. 
Under this system, we have
\begin{equation}\label{eq:xphi>E[LXphi]}
||x^*||_1\,\varphi(t^*,x^*) 
- \mathbb{E} \left[L(\rho)X(\rho)\varphi \big(t^*-\rho,\mathfrak{X}(\rho)\big)\right] 
=\, \mathbb{E} \left[\int_0^{\rho} L(s)X(s)\, g\big(t^*-s,s,\mathfrak{X}\big)\, \mathrm{d}s \right] >0\,,
\end{equation}
for any sufficiently small positive stopping time $\rho$\,, where 
$$g(t,s,\mathfrak{X}) := (\varphi_t-\mathcal{L}_{\alpha(s,\mathfrak{X})}\varphi)\left(t,\mathfrak{X}(s)\right) \ge \big(\varphi_t-\widehat{\mathcal{L}}\varphi\big)\left(t,\mathfrak{X}(s)\right)$$ for $(t,s) \in (0, \infty)\times [0, \infty)$ and with 
$\mathcal{L}_{a}$ and $\widehat{\mathcal{L}}$  as in \eqref{eq:hatL}. This displayed quantity   is positive for any sufficiently small $s$, and $t$ sufficiently close to $t^*$, by virtue of \eqref{eq:hatG>0} and the continuity of the function $F$ in \eqref{eq:ourF}.

\smallskip
On the other hand, on the left-hand side of \eqref{eq:xphi>E[LXphi]}  we can estimate $\varphi(t^*,x^*)$ from above by $\mathfrak{u}_{\mathcal{M}^{x^*}}(t^*,x^*)+2\,C_3$\,, and $\varphi \big(t^*-\rho,\mathfrak{X}(\rho)\big)$ from below by $\widehat{\Phi}\big(t^*-\rho,\mathfrak{X}(\rho)\big)+C_2$ (for some $\omega$'s in $\Omega$) or by  $\widehat{\Phi}\big(t^*-\rho,\mathfrak{X}(\rho)\big)$ (for other $\omega$'s in $\Omega$)  with $C_2$ a small positive constant; this allows us to deduce  
\begin{equation}\label{eq:xu>E[LXPhi]}
||x^*||_1 \, \mathfrak{u}_{\mathcal{M}^{x^*}}(t^*,x^*)> \mathbb{E} \left[L(\rho)X(\rho)\, \widehat{\Phi}\big(t^*-\rho,\mathfrak{X}(\rho)\big)\right].
\end{equation} 
But the inequality \eqref{eq:xu>E[LXPhi]} turns out to contradict the martingale property   of the process $$L(\cdot)X(\cdot) \, \mathfrak{u}_{\mathcal{M}^{\mathfrak{X}(\cdot)}}\big(T-\cdot\,,\mathfrak{X}(\cdot)\big) ;$$ 
see Proposition \ref{prop:martingale}, and recall $\mathcal{M}^x$, $\,x\in \mathbb{R}_+^n\,$ from  Definition \ref{def:MM}.
\smallskip

When implementing this program, the stopping time $\rho$ needs to be not only   small, but also such  that on $[0,\rho]$  the processes
$L(\cdot)$ and $\mathfrak{X}(\cdot)$ are bounded,   and $\mathfrak{X}(\cdot)$  is close to  $x^*$; however,   $\rho$ cannot be too small, in order to ensure that $\varphi \big(t^*-\rho,\mathfrak{X}(\rho)\big)\ge \widehat{\Phi}\big(t^*-\rho,\mathfrak{X}(\rho)\big)+C_2$ holds with a probability greater than some positive constant independent of $C_2$ ($1/2$ in the following proof, see Lemma \ref{lemma:rho=nu}). These considerations inspire us to construct  $\rho$ as in  \eqref{eq:nu}--\eqref{eq:rho} below.

\smallskip

\begin{proof}[Proof of Theorem \ref{thm:viscosity1}:]
According to Definition \ref{def:viscosity sol}\,(i) of viscosity subsolution with the function $F$ as in \eqref{eq:ourF}, it suffices to show that for any test function $\varphi \in C^{1,2}\left( (0,\infty) \times \mathbb{R}_{+}^n\right) $ and  $(t_0,x_0) \in (0,\infty) \times \mathbb{R}_{+}^n$ with 
\begin{equation}\label{eq:max}
\big(\widehat{\Phi}^* - \varphi\big)(t_0,x_0) = 0 > \big(\widehat{\Phi}^* - \varphi\big)(t,x)\,, \ \ \forall\  (t,x)\in (0,\infty) \times \mathbb{R}_{+}^n\,,
\end{equation}
(i.e., such that $(t_0,x_0)$ is a strict maximum of $ \, \widehat{\Phi}^* - \varphi$), we have  $$\big(\varphi_t-\widehat{\mathcal{L}}\varphi\big)(t_0,x_0) \le 0\,.$$
Here $\widehat{\mathcal{L}}$ is defined in \eqref{eq:hatL}, and $\, \widehat{\Phi}^*\,$\label{p:hatPhi^*} is the upper-semicontinuous envelope of  $\, \widehat{\Phi}\,$ as in the definition \eqref{eq:u^*}. {\it We shall argue this by contradiction, assuming that} 
\begin{equation}\label{eq:hatG}
\widehat{\mathcal{G}}(t_0,x_0)>0\, \quad \text{holds for the function} \qquad \widehat{\mathcal{G}}(t,x) := \big(\varphi_t-\widehat{\mathcal{L}}\varphi\big)(t,x)\,.
\end{equation}
Since the function $F$ of \eqref{eq:ourF} is continuous, so is the function $\widehat{\mathcal{G}}$  just introduced in  \eqref{eq:hatG}. There will exist then,  under this hypothesis and Assumption  \ref{asmp1}, a neighborhood $\,\mathcal{D}_{\delta} := (t_0 - \delta, t_0 + \delta) \times B_{\delta}(x_0)$\label{p:D_delta} of $(t_0,x_0)$ in $(0,\infty) \times \mathbb{R}_{+}^n$ with $0< \delta < ||x_0||_1 / n\,$, on which    $\mathcal{K} (\cdot)$ is bounded and $\widehat{\mathcal{G}} (\cdot\,, \cdot) >0$ holds.

\medskip  
Let $C$\label{p:C} be a constant such that $||\theta||<C$ and $|a_{ij}|<C$ $(1\le i,j \le n)$ hold for all pairs $(\theta, a=(a_{ij})_{n \times n})\in \mathcal{K}(x)$ and all $x\in B_{\delta}(x_0)$\,. We notice that \label{p:(x_0)_i}$|x_i-(x_0)_i|\le |x-x_0|<\delta$ holds for any $x=(x_1,\dots,x_n)\in \mathcal{D}_{\delta}$, thus
\begin{equation}\label{eq:||x||_1}
0 < ||x_0||_1 - n\delta < ||x||_1< ||x_0||_1 + n\delta \,,
\end{equation}
and introduce the strictly positive constants 
\begin{equation}\label{eq:C_1}
C_1:=\sqrt{32\, \delta C^2 + 4\, \delta^2 C^4}\,,\  \ \ \ \ 
C_2:= - \max_{\partial \mathcal{D}_{\delta}} \, \big(\widehat{\Phi}^* - \varphi \big)(t,x)\,, \ \ \ \ \ \ 
C_3:=\frac{C_2\,  e^{-C_1}(||x_0||_1-n\delta)}{4(||x_0||_1 + n\delta)}
\end{equation}
(the positivity of $C_2$ and $C_3$ follows from  \eqref{eq:max} and \eqref{eq:||x||_1}, respectively).
We observe that 
$$\limsup_{(t,x)\to (t_0,x_0)}\big(\widehat{\Phi}-\varphi\big)(t,x) = \big(\widehat{\Phi}^*-\varphi\big)(t_0,x_0)=0\,,$$ 
hence there exists $(t^*,x^*)\in \mathcal{D}_{\delta}$ such that  
\begin{equation}\label{eq:t*,x*}
\big(\widehat{\Phi}-\varphi\big)(t^*,x^*) > -C_3\,;
\end{equation}
and by the definition \eqref{eq:Phi} of $\widehat{\Phi}$, there exists an admissible system $\mathcal{M}^{x^*}\in \widehat{\mathfrak{M}}(x^*)$ such that 
\begin{equation}\label{eq:u_M^x*}
\mathfrak{u}_{\mathcal{M}^{x^*}}(t^*,x^*) 
> \widehat{\Phi}(t^*,x^*) - C_3>\varphi(t^*,x^*) - 2\,C_3\,, \quad{\mathrm{by}}\ \eqref{eq:t*,x*}\,. 
\end{equation}
The remaining discussion in this section (with the exception of Proposition \ref{prop:martingale}) will be carried out under this admissible system. 
 
\medskip
\noindent
$\bullet~$ Let us start by recalling the definitions of $\mathcal{D}_{\delta}$ and $ t^*$, and by constructing the positive stopping times
\begin{equation}\label{eq:nu}\ \ 
\nu \,(=\nu(\omega)):=\inf \big\{s\in (0, t^*] : \big(t^*-s,\mathfrak{X}(s)\big)\notin \mathcal{D}_{\delta} \big\} \le t^*-(t_0-\delta)=(t^*-t_0)+\delta<t^* \wedge 2\delta\,,
\end{equation}
\begin{equation}\label{eq:rho}
\lambda \,(=\lambda \,(\omega)):=\inf\{s>0: |\log L(s)| > C_1\}\,, \qquad \rho \,(=\rho(\omega)):=\nu \wedge \lambda 
\end{equation}
with the usual convention inf$\,\emptyset = \infty$\,. From the definitions \eqref{eq:hatG} and \eqref{eq:hatL}, we see that  
\begin{equation}\label{eq:g}
g(t,s,\mathfrak{X}) := \left( \varphi_t-\mathcal{L}_{\alpha(s,\mathfrak{X})}\varphi\right) \big(t,\mathfrak{X}(s)\big)
\ge \widehat{\mathcal{G}}\big(t,\mathfrak{X}(s)\big)\,,\quad \forall\ (t,s) \in (0, \infty)\times [0, \infty)\,.
\end{equation}
Recall that $\widehat{\mathcal{G}} (\cdot\,, \cdot) >0$ holds on $\mathcal{D}_{\delta}$, from the discussion right below \eqref{eq:hatG}. Combining with \eqref{eq:g}, this observation leads to
\begin{equation}\label{eq:g>0}
g(t^*-s,s,\mathfrak{X}) >0\,,\quad \forall\ s \in [0, \rho)\,.
\end{equation}

\smallskip
Thanks to the assumption $\varphi \in C^{1,2}\left( (0,\infty)\times \mathbb{R}_+^n\right) $, we can apply \textsc{It\^o}'s change of variable rule  to $X(t)L(t)\varphi(T-t,\mathfrak{X}(t))$, $0 \le t \le T$  and derive the following decomposition (see Appendix \ref{appendix:Ito} for a detailed proof).

\begin{lemma}\label{lemma:Ito}
For any $0\leq t<T<\infty$, $x\in \mathbb{R}_+^n$\,, $\varphi \in C^{1,2}\left( (0,\infty)\times \mathbb{R}_+^n\right) $,  and diffusion $\mathfrak{X}(\cdot)$ satisfying \eqref{eq:SDE},
we have

\begin{align}
\nonumber
\mathrm{d}\big(L(t)X(t)\varphi\big(T-t,\mathfrak{X}(t)\big) \big) 
\,= \,&-L(t)X(t) g\big(T-t,t,\mathfrak{X}\big)\, \mathrm{d}t 
\\ 
\nonumber
&-X(t)\varphi\big(T-t,\mathfrak{X}(t)\big)L(t)\,\vartheta'(t,\mathfrak{X})\,  \mathrm{d}W(t)\\
\label{eq:Ito}
&+L(t)\sum_{i,k} X_i(t)  \big[\varphi+X(t) D_i \varphi\big] \big(T-t,\mathfrak{X}(t)\big) \sigma_{ik}(t,\mathfrak{X})\, \mathrm{d}W_k(t)\,.
\end{align}
\end{lemma}

Let us apply now Lemma \ref{lemma:Ito} with $T=t^*$, integrating \eqref{eq:Ito} with respect to $t$ over $[0, \rho]$  and taking the expectation under $\mathbb{P}$, to obtain 
\begin{equation}
\label{eq:E[LXphi]}\ 
||x^*||_1\,\varphi(t^*,x^*) 
- \mathbb{E} \left[L(\rho)X(\rho)\varphi \big(t^*-\rho,\mathfrak{X}(\rho)\big)\right] 
=\, \mathbb{E} \left[\int_0^{\rho} L(s)X(s)\, g\big(t^*-s,s,\mathfrak{X}\big)\, \mathrm{d}s \right]  >\, 0\,.
\end{equation}
Here, the strict inequality comes from \eqref{eq:g>0} and  the positivity of $\rho$\,; whereas,  in the equality, the expectations of the integrals with respect to $\mathrm{d}W(t)$ or $\mathrm{d}W_k(t)$ have all vanished. This is due to the  the boundedness of the processes $\mathfrak{X}(\cdot)$ and $L(\cdot)$ on $[0,\rho]$, of the functions $\varphi$ and $D_i \varphi$ on $\overline{\mathcal{D}_{\delta}}\,$, and of the functionals $\vartheta(\cdot, \mathfrak{X})$, $\alpha_{ij}(\cdot, \mathfrak{X})$ (by Assumption \ref{asmp1}) and thus $\sigma_{ik}(\cdot, \mathfrak{X})$ on $[0,\rho]$.

\smallskip
(We have made use here of the following facts.  The eigenvalues $e_i$\label{p:e_i} of $\alpha$ are the nonnegative roots of the characteristic polynomial of $\alpha$, which is determined by the entries $\alpha_{ij}$; since the $\alpha_{ij}(\cdot, \mathfrak{X})$'s are bounded on $[0,\rho]$, so are the $e_i$'s. Thus $\sigma$, which can be written as $\mathbf{QD}$\label{p:QD} for some $n\times n$ orthonormal matrix $\mathbf{Q}$ and diagonal matrix $\mathbf{D}$ with diagonal entries $\sqrt{e_i}$\,, is also bounded.)

\smallskip
Notice that $\,\big(t^*-\nu,\mathfrak{X}(\nu)\big) \in \partial\mathcal{D}_{\delta}\,$ holds by the definition \eqref{eq:nu} of $\nu$, so  
we have 
\begin{equation}\label{eq:partialD}
\varphi \big(t^*-\nu,\mathfrak{X}(\nu)\big)
\ge \widehat{\Phi}^* \big(t^*-\nu,\mathfrak{X}(\nu)\big) +C_2
\ge \widehat{\Phi} \big(t^*-\nu,\mathfrak{X}(\nu)\big) +C_2\,.
\end{equation}
Plugging (\ref{eq:max}), \eqref{eq:u_M^x*} and (\ref{eq:partialD}) into (\ref{eq:E[LXphi]}) yields  
\begin{eqnarray}
\nonumber
&&\ 0 <||x^*||_1\, \big[\, \mathfrak{u}_{\mathcal{M}^{x^*}}(t^*,x^*)+2\,C_3 \, \big]
- \mathbb{E} \left[\bm{1}_{\{\rho=\nu\}}L(\rho)X(\rho)\left(\widehat{\Phi} \big(t^*-\rho,\mathfrak{X}(\rho)\big)+C_2\right)\right. \\
\label{eq:>0}
&&\ \ \ \ \ \ \ + \left.\bm{1}_{\{\rho\neq\nu\}}L(\rho)X(\rho)\,\widehat{\Phi} \big(t^*-\rho,\mathfrak{X}(\rho)\big)\right] \\ \nonumber
&&\ \ \ =
||x^*||_1\,\mathfrak{u}_{\mathcal{M}^{x^*}}(t^*,x^*) 
- \mathbb{E} \left[L(\rho)X(\rho)\,\widehat{\Phi} \big(t^*-\rho,\mathfrak{X}(\rho)\big)\right]
+ 2\,C_3\, ||x^*||_1  
- C_2\, \mathbb{E} \big[\bm{1}_{\{\rho=\nu\}}L(\rho)X(\rho)\big]\, .
\end{eqnarray}

We start by estimating the last term on the right-hand side of \eqref{eq:>0}. 
Recalling the definition \eqref{eq:rho} of $\rho$ and the second inequality in \eqref{eq:||x||_1}, we see that
\begin{equation}\label{eq:L(rho)}
 L(\rho) \ge e^{-C_1}\ \ \ \ {\mathrm{and}}\ \ \ \ X(\rho)>||x_0||_1 - n\delta\, >0\,, 
\end{equation} 
hence
\begin{equation}\label{eq:E[1LX]}
\mathbb{E} \big[\bm{1}_{\{\rho=\nu\}}L(\rho)X(\rho)\big]
\ge e^{-C_1}\big(||x_0||_1 - n\delta \big)\,\, \mathbb{P} \big(\rho=\nu\big)\,.
\end{equation} 

\begin{lemma}
\label{lemma:rho=nu}
We have
\begin{equation}\label{eq:rho=nu}
\mathbb{P}\big( \rho =\nu\big) \,=\, \mathbb{P}\big(\lambda\ge \nu\big)  \ge \frac{1}{\,2\,}\,.
\end{equation}
\end{lemma}
\begin{proof}
For any $t\in(0, \nu]\,,$ we have 
\begin{eqnarray*}
\big( \log  L (t) \big)^2 &= &\left|-\int_0^t\vartheta'(s,\mathfrak{X})\, \mathrm{d}W(s) - \int_0^t\frac{1}{\,2\,} \,\big|\big|\vartheta(s,\mathfrak{X})\big|\big|^2\, \mathrm{d}s\right|^2\\
&\le& 2\left|\int_0^t\vartheta'(s,\mathfrak{X})\, \mathrm{d}W(s)\right|^2
+ 2\left|\int_0^t\frac{1}{\,2\,} \,\big|\big|\vartheta(s,\mathfrak{X})\big|\big|^2\, \mathrm{d}s\right|^2.
\end{eqnarray*}
It follows from $t\le \nu<2\delta$ that
\begin{equation*}
\int_0^t \frac{1}{\,2\,} \,\big|\big|\vartheta(s,\mathfrak{X})\big|\big|^2\, \mathrm{d}s
\, \le \, \frac{t}{\,2\,} \,C^2 \, 
\le\, \delta C^2,
\end{equation*}
and therefore
\begin{equation*} 
\mathbb{E}\left[\sup_{0\le t\le \nu} \big( \log  L (t) \big)^2\right] 
\le 2\, \mathbb{E}\left[\sup_{0\le t\le \nu} \left|\int_0^t\vartheta'(s,\mathfrak{X})\, \mathrm{d}W(s)\right|^2\right]
+ 2\, \delta^2 C^4.
\end{equation*}
Further, the \textsc{Burkholder-Davis-Gundy} Inequality gives
$$
2\, \mathbb{E}\left[\sup_{0\le t\le \nu} \left|\int_0^t\vartheta'(s,\mathfrak{X})\, \mathrm{d}W(s)\right|^2\right]
\le 8\, \mathbb{E} \left[\int_0^{\nu}\left|\left|\vartheta'(s,\mathfrak{X})\right|\right|^2\mathrm{d}s\right]
\le 8\, \mathbb{E} \left[\nu C^2\right]
\le 16\, \delta C^2,
$$
thus
\begin{equation*}
\mathbb{E}\left[\sup_{0\le t\le \nu} \big( \log  L (t) \big)^2\right] 
\le  16\, \delta C^2 + 2\, \delta^2 C^4.
\end{equation*}
Finally, appealing to \textsc{Markov}'s Inequality yields
\begin{equation*}
\mathbb{P}\big(\lambda< \nu\big)
= \,\mathbb{P}\left[\sup_{0\le t\le \nu} |\log L (t)|> C_1\right] 
\le \,\frac{16\, \delta C^2 + 2\, \delta^2 C^4}{C_1^2} = \frac{1}{\,2\,} 
\end{equation*}
(this is why we defined $C_1$ as   in \eqref{eq:C_1}; in fact, setting $C_1$ to be any value greater than the right-hand side of the first equation in \eqref{eq:C_1} would also work), and the claim (\ref{eq:rho=nu}) follows. 
\end{proof}

Substituting the estimate of Lemma \ref{lemma:rho=nu} into \eqref{eq:E[1LX]},  we obtain
\begin{equation*} 
C_2\,\mathbb{E} \big[\bm{1}_{\{\rho=\nu\}}L(\rho)X(\rho)\big]
\ge \frac{1}{\,2\,} C_2\, e^{-C_1}\big(||x_0||_1 - n\delta \big)=2\,C_3  \big( ||x_0||_1 + n\delta \big) >2\,C_3 \,||x^*||_1\,,
\end{equation*} 
where we used  the definition \eqref{eq:C_1} of $C_3$ and the last inequality in \eqref{eq:||x||_1}. 
Plugging into \eqref{eq:>0} yields the inequality \eqref{eq:xu>E[LXPhi]}; 
however, this inequality contradicts Proposition \ref{prop:martingale}\,(ii) right below with $T=t^*$ and $\tau=\rho$\,.  (This explains why we constructed $C_3$ as we did in \eqref{eq:C_1}; in fact, setting $C_3$ to be any value less than the right-hand side of \eqref{eq:C_1} would also work.)
 
 The proof of Theorem 4.5 is complete.
\end{proof}

\begin{proposition}\label{prop:martingale} 
{\bf Martingale Property:}  Recall the strongly Markovian admissible systems $\mathcal{M}^{y}\in \widehat{\mathfrak{M}}(y)$ $(y\in \mathbb{R}_+^n)$ from Definition \ref{def:MM}.  

\smallskip
\noindent
{\bf (i)}
For $0\leq t\leq T<\infty$, we have
\begin{equation*}
L(t)X(t) \, \mathfrak{u}_{\mathcal{M}^{\mathfrak{X}(t)}}\big(T-t,\mathfrak{X}(t)\big) = \,\mathbb{E} \big[\,L(T)X(T) \,\big|\, \mathcal{F}(t)\, \big]\,,\ \mathbb{P} \mathrm{-a.s.}
\end{equation*}
In particular, the process on the left-hand side is a martingale.

\smallskip

\noindent
{\bf (ii)} For any stopping time $\tau\leq T<\infty$, we have 
\begin{equation*}
\mathbb{E} \left[L(\tau)X(\tau)\, \mathfrak{u}_{\mathcal{M}^{\mathfrak{X}(\tau)}}\big(T-\tau,\mathfrak{X}(\tau)\big)\right] 
= ||x^*||_1 \, \mathfrak{u}_{\mathcal{M}^{x^*}}(T,x^*)\,.
\end{equation*}
\end{proposition}

\begin{proof}
\textit{(i)} To alleviate notation somewhat, we write $\mathbb{P}^{\,y}$, $W^{y}(\cdot)$, $X^{y}(\cdot)$\label{p:X^y} and $L^{y}(\cdot)$ for $\mathbb{P}^{\,\mathcal{M}^{y}}$, $W^{\mathcal{M}^{y}}(\cdot)$ $X^{\mathcal{M}^{y}}(\cdot)$ and $L^{\mathcal{M}^{y}}(\cdot)$ $(y\in \mathbb{R}_+^n)$, respectively. The definitions \eqref{eq:um} of $\mathfrak{u}_{\mathcal{M}}$ and \eqref{eq:L=exp} of $L^{\mathcal{M}}$ give 
\begin{align*}
\mathrm{LHS}
=&\ L(t)\, \mathbb{E}^{\mathbb{P}^{\,\mathfrak{X}(t)}}\left[X^{\mathfrak{X}(t)}(T-t)L^{\mathfrak{X}(t)}(T-t)\right]\\
=&\ L(t)\, \mathbb{E}^{\mathbb{P}^{\,\mathfrak{X}(t)}}\left[X^{\mathfrak{X}(t)}(T-t)\exp\left(-\int_0^{T-t}\vartheta'\left(s,\mathfrak{X}^{\mathfrak{X}(t)}\right)\, \mathrm{d}W^{\mathfrak{X}(t)}(s)- \int_0^{T-t} \frac{1}{\,2\,} \,\big|\big|\vartheta\left(s,\mathfrak{X}^{\mathfrak{X}(t)}\right)\big|\big|^2 \mathrm{d}s\right)\right]\\
=&\ L(t)\, \mathbb{E}\left[ \left. X(T) \exp\left(-\int_t^T\vartheta'(s,\mathfrak{X})\, \mathrm{d}W(s)  - \int_t^T \frac{1}{\,2\,} \,\big|\big|\vartheta(s,\mathfrak{X})\big|\big|^2 \mathrm{d}s\right)\right|\mathcal{F}(t)\right]\\
=&\ L(t)\, \mathbb{E}\big[X(T)L(T)\,/\,L(t)\,|\,\mathcal{F}(t)\big]
= \mathrm{RHS}\,,\ \mathbb{P}\mathrm{-a.s.} 
\end{align*}
We note that in the third equality we took advantage of \eqref{eq:Markovian} and of the strong Markov property for the process  $\mathfrak{X}(\cdot)$.

\smallskip
\textit{(ii)} On the strength of the martingale property from (i), the Optional Sampling Theorem gives 
$\mathrm{LHS}
= L(0)X(0)\mathfrak{u}_{\mathcal{M}^{x^*}}(T,\mathfrak{X}(0))
= \mathrm{RHS}\,.$
\end{proof}

\begin{remark}
{\rm 
In the above proof of Theorem  \ref{thm:viscosity1}, the special structure of strongly Markovian admissible systems that we selected in Definition \ref{def:MM}, is indispensable in the context of Proposition \ref{prop:martingale}. On the other hand, the Assumption \ref{asmp1} is important for the existence of the neighborhood  $\mathcal{D}_{\delta}\,$ with the stated properties; see the discussion right below \eqref{eq:hatG}.
}  \qed
\end{remark}

\section{Proof of Theorem \ref{thm:viscosity2}: Viscosity Supersolution }
\label{section:supersol}

The  proof that follows shares many similarities with that    in Section \ref{section:subsol} 
for Theorem \ref{thm:viscosity1}, the counterpart of Theorem \ref{thm:viscosity2}, but also requires the additional Assumption \ref{asmp2} and a much stronger result -- the {\it  Dynamic Programming Principle} (or DPP, Proposition 6.1 below) --  than the martingale property of Proposition \ref{prop:martingale}. Before outlining and presenting the proof,  we explain the reasons for such differences.
 
\medskip
 \noindent
 $\bullet~$ 
 We begin with an idea    similar to that in Section \ref{section:subsol} (with corresponding inequalities in opposite directions, and with $\widehat{\Phi}$ replaced by $\Phi$); however, we cannot proceed in the same way for two reasons:

\smallskip

{\bf (i)} The reverse inequality to \eqref{eq:g}, namely,  
$g(t,s,\mathfrak{X})  \le  \widehat{\mathcal{G}}\big(t,\mathfrak{X}(s)\big)$ 
does not hold in general,  by the definition \eqref{eq:hatL} of $\widehat{\mathcal{L}}$ (recall $g$ from \eqref{eq:g} and $\widehat{\mathcal{G}}$ from \eqref{eq:hatG}). 
Therefore, we cannot obtain 
\begin{equation}\label{eq:g<0}
g(t^*-s,s,\mathfrak{X})<0\quad {\mathrm{for\ all\ }} s\in [0,\rho)\,,\quad {\mathrm{with}}\ g\ {\mathrm{as\ in\ }} \eqref{eq:g} \ {\mathrm{and\ }}  \rho\ {\mathrm{as\ in\ }}  \eqref{eq:rho}\,,
\end{equation}
the reverse inequality to \eqref{eq:g>0}, as we did in Section \ref{section:subsol}. Instead, we need to find an admissible system in ${\mathfrak{M}}(x^*)$ under which \eqref{eq:g<0} holds.

If we still want to argue by contradiction, assuming the reverse inequality to  \eqref{eq:hatG}, then according to the definitions \eqref{eq:hatL} of $\mathcal{L}$ and \eqref{eq:hatG} of $\widehat{\mathcal{G}}$, there exists $a_0\in \mathcal{A}(x_0)$ such that
$
(\varphi_t-{\mathcal{L}_{a_0}}\varphi)(t_0,x_0)<0\,.
$
Plugging in the definition  \eqref{eq:hatL} of $\mathcal{L}_{a_0}$ and comparing the left-hand side of this inequality with the  $g(t^*-s,s,\mathfrak{X})$ of \eqref{eq:g}, 
we see that \eqref{eq:g<0} holds if the $\alpha$ in \eqref{eq:g} is very close to $a_0$ when $s$ is sufficiently small. This accounts for the requirement \eqref{eq:continuous a} of Assumption \ref{asmp2}. Other conditions in Assumption \ref{asmp2} are inspired by Remark \ref{rmk:nonemptyM} aimed for the existence of an admissible system with such $\alpha$.

\smallskip

{\bf (ii)} The reverse inequality of \eqref{eq:xu>E[LXPhi]} with $\widehat{\Phi}$ replaced by $\Phi$, namely  
\begin{equation} \label{eq:xu<E[LXPhi]}
||x^*||_1 \, \mathfrak{u}_{\mathcal{M}^{x^*}}(t^*,x^*)< \mathbb{E} \left[L(\rho)X(\rho)\, {\Phi}\big(t^*-\rho,\mathfrak{X}(\rho)\big)\right],
\end{equation} 
actually holds in general, on the strength of Proposition \ref{prop:martingale} and the definition \eqref{eq:Phi} of ${\Phi}$. Therefore we need to estimate more accurately the value of $\varphi$ on the left-hand side of the counterpart of \eqref{eq:xphi>E[LXphi]},  by using ${\Phi}$ instead of $\mathfrak{u}_{\mathcal{M}^{x^*}}$, so that we arrive at
 \begin{equation} \label{eq:xPhi<E[LXPhi]}
||x^*||_1 \, {\Phi}(t^*,x^*)< \mathbb{E} \left[L(\rho)X(\rho)\, {\Phi}\big(t^*-\rho,\mathfrak{X}(\rho)\big)\right],
\end{equation} 
instead of \eqref{eq:xu<E[LXPhi]}. We then need  the DPP of Proposition \ref{prop:DPP},  to obtain a contradiction to \eqref{eq:xPhi<E[LXPhi]}.

\subsubsection{Informal Outline}

Now we outline the main steps of the proof. 
We   prove  by contradiction, assuming the negation of \eqref{eq:F>=0} in Definition \ref{def:viscosity sol}  with the function $F$ as in \eqref{eq:ourF},  that there exist   $\varphi \in C^{1,2}\left( (0,\infty) \times \mathbb{R}_{+}^n\right)$ and  $(t_0,x_0) \in (0,\infty) \times \mathbb{R}_{+}^n$ such that:  $(t_0,x_0)$ is a strict minimum of $ \, {\Phi}_* - \varphi \,$; the minimal value is equal to zero; and   $\big(\varphi_t-\widehat{\mathcal{L}}\varphi\big)(t_0,x_0)<0$\,. 

Since $\widehat{\mathcal{L}}\varphi= \sup_{a\in  \mathcal{A}(x)}\mathcal{L}_{a} \varphi$ (definition \eqref{eq:hatL}),  there exists $a_0\in \mathcal{A}(x_0)$ such that
\begin{equation}\label{eq:G_a<0}
(\varphi_t-{\mathcal{L}_{a_0}}\varphi)(t_0,x_0)<0\,.
\end{equation}
We take $(x,a)=(x_0,a_0)$ and a sufficiently small $\iota$ in Assumption \ref{asmp2}, and let $\delta$, $\mathbf{s}$ and $\boldsymbol{\theta}$ be the corresponding elements. 
Further, by the definition \eqref{eq:u_*} of ${\Phi}_*$\,, we can take a pair $(t^*,x^*)$ close to $(t_0,x_0)$ such that the nonnegative difference $(\Phi-\varphi)(t^*,x^*)$ is sufficiently small, say less than a small positive constant $C_3^*$ (depending on $\delta$; defined similarly to the $C_3$ of \eqref{eq:C_1}). 

Thanks to Assumption \ref{asmp2},  there exists an admissible system $\mathcal{M}^{x^*}\in \mathfrak{M}(x^*)$ with the functionals $\sigma$ and $\vartheta$ defined by \eqref{eq:Markovian}.
Under this admissible system, we derive \eqref{eq:g<0} from \eqref{eq:G_a<0}, and thus
\begin{equation}\label{eq:xphi<E[LXphi]}
||x^*||_1\,\varphi(t^*,x^*) - \mathbb{E} \left[L(\rho)X(\rho)\varphi \big(t^*-\rho,\mathfrak{X}(\rho)\big)\right] 
= \mathbb{E} \left[\int_0^{\rho} L(s)X(s)\, g\big(t^*-s,s,\mathfrak{X}\big)\, \mathrm{d}s \right] <0\,.
\end{equation}

On the other hand, on the left-hand side of \eqref{eq:xphi<E[LXphi]}   we estimate the real number $\varphi(t^*,x^*)$ from below by ${\Phi}(t^*,x^*)-C_3^*$\,, and the random quantity $\varphi \big(t^*-\rho,\mathfrak{X}(\rho)\big)$ from above by ${\Phi}\big(t^*-\rho,\mathfrak{X}(\rho)\big)-C_2^*$ (for some $\omega$'s in $\Omega$) or  ${\Phi}\big(t^*-\rho,\mathfrak{X}(\rho)\big)$ (for other $\omega$'s in $\Omega$) with $C_2^*$ a small positive constant similar to the $C_2$ of \eqref{eq:C_1}, and then deduce \eqref{eq:xPhi<E[LXPhi]}, which  contradicts the Dynamic Programming Principle   of Proposition \ref{prop:DPP}.
\qed

\subsection{The Supersolution Property}

We are ready now to present the argument  proper. 

\begin{proof}[Proof of Theorem \ref{thm:viscosity2}:]
According to  Definition \ref{def:viscosity sol}\,(ii)  of viscosity supersolution with the function $F$ as in \eqref{eq:ourF}, it suffices to show that for any test function $\varphi \in C^{1,2}\left((0,\infty) \times \mathbb{R}_{+}^n\right)$  and $(t_0,x_0) \in (0,\infty) \times \mathbb{R}_{+}^n$ with 
\begin{equation}\label{eq:min}
({\Phi}_* - \varphi)(t_0,x_0) = 0 < ({\Phi}_* - \varphi)(t,x)\,, \ \ \forall\  (t,x)\in (0,\infty) \times \mathbb{R}_{+}^n 
\end{equation}
 (i.e., such that $(t_0,x_0)$ is a strict minimum of $\,  {\Phi}_* - \varphi  $),
and with $\, {\Phi}_*\,$\label{p:hatPhi_*} the lower-semicontinuous envelope of  $\,  {\Phi} \,$   as in the definition \eqref{eq:u_*}, we have 
$$\big( \varphi_t-\widehat{\mathcal{L}}\varphi\big)  (t_0,x_0)   \ge 0\,.$$
Recalling $\widehat{\mathcal{L}}$ from  \eqref{eq:hatL},  it suffices to establish $(\varphi_t-{\mathcal{L}}_{a}\varphi)  (t_0,x_0)  \ge 0$ for every fixed $a\in \mathcal{A}(x)$. 

\smallskip
{\it  We shall argue this by contradiction, assuming that for some $a_0\in \mathcal{A}(x_0)$ we have }
\begin{equation}\label{eq:G_a}\ g_0:=-\,\mathcal{G}_{a_0}(t_0,x_0)>0\,,\ \ \ \mathrm{where}\ \ \ 
\mathcal{G}_{a}(t,x):=( \varphi_t-{\mathcal{L}}_{a}\varphi) (t,x)\,, \ (a,t,x)\in \mathcal{A}(x) \times (0,\infty)\times \mathbb{R}_+^n\,.
\end{equation}

Under Assumption \ref{asmp1}, there exists a positive number $\delta_1<t_0 \wedge (||x_0||_1 / n)$ such that 
$\mathcal{K}(\cdot)$ is bounded on $\mathcal{D}_{\delta_1}:=(t_0-\delta_1, t_0+\delta_1)\times B_{\delta_1}(x_0)$. 
Let $C>1$ be a constant such that $||\theta||,|a_{ij}|<C$ $(1\le i,j \le n)$ hold for all pairs $(\theta, a=(a_{ij})_{n \times n})\in \mathcal{K}(x)$ and all $x\in B_{\delta_1}(x_0)$.

Since the functions $\mathcal{G}_{a_0}(\cdot,\cdot)$, $\varphi_t(\cdot,\cdot)$ and
\begin{equation}\label{eq:H_{ij}}
H_{ij}(s,y):= D_i\varphi(s,y)\,/\,||y||_1
+ y_iy_jD^2_{ij}\varphi(s,y)\,/\,2\,, \quad (s,y)\in (0,\infty)\times \mathbb{R}_+^n\,,\ 1\le i,j\le n
\end{equation} 
are continuous, there exists under the hypothesis \eqref{eq:G_a},  a positive number $\delta_2<\delta_1$ such that for all $H\in\{\varphi_t, H_{ij}\, (1\le i,j\le n)\}$, we have
\begin{equation}\label{eq:H}
|H(t,x)-H(t_0,x_0)|< g_0\,/\,3n^2 C<g_0\,/\,3\,, \quad  \forall\ (t,x)\in \mathcal{D}_{\delta_2}
:=(t_0-\delta_2, t_0+\delta_2)\times B_{\delta_2}(x_0)\,.
\end{equation}

\begin{lemma}
With $\mathcal{G}_{a}(\cdot,\cdot)$ defined in \eqref{eq:G_a}, the inequality 
\begin{equation}\label{eq:Ga-Ga0}
\left|\mathcal{G}_{a}(t,x)-\mathcal{G}_{a_0}(t_0,x_0)\right| <g_0
\end{equation}
 holds for all $\,(t,x)\in \mathcal{D}_{\delta_2}\,,\ 
a\in \mathcal{A}(x)\,$ with 
\begin{equation}\label{eq:iota}  
\max_{1\le i,j\le n} |a_{ij}-(a_0)_{ij}| < \iota := \delta_2\wedge g_0  \left( 1+3n^2\max_{i,j} \left|H_{ij}(t_0,x_0)\right|\right) ^{-1}.
\end{equation}
Recalling the number $g_0$ from the definition \eqref{eq:G_a}, we have also 
\begin{equation}\label{eq:G<0}
\mathcal{G}_{a}(t,x)<0\quad {\mathrm{for\ all}}\   (a,t,x)\ {\mathrm{in}}\ \eqref{eq:iota}.
\end{equation}
\end{lemma}
 
\begin{proof}
Plugging  the definition  \eqref{eq:G_a} of ${\mathcal{G}_{a}}$ into the left-hand side of \eqref{eq:Ga-Ga0} yields
\begin{eqnarray}\label{eq:Lambda}
\nonumber {\mathrm{LHS\ of}}\ \eqref{eq:Ga-Ga0} 
&= &  \left|(\varphi_t-{\mathcal{L}_{a}}\varphi)(t,x) - \left( \varphi_t-{\mathcal{L}}_{a_0}\varphi\right) (t_0,x_0)\right| \\ 
&\le& \left|  \varphi_t(t,x) - \varphi_t (t_0,x_0)\right| 
+\left| {\mathcal{L}}_{a}\varphi(t_0,x_0) - {\mathcal{L}}_{a}\varphi(t,x)\right|
+ \left| {\mathcal{L}}_{a_0}\varphi(t_0,x_0)- {\mathcal{L}}_{a}\varphi(t_0,x_0)\right| \\ \nonumber
&=:& \Lambda_1+\Lambda_2+\Lambda_3\,,
\end{eqnarray}
i.e., $\Lambda_j$ $(j = 1, 2, 3)$ denotes the $j$-th term in \eqref{eq:Lambda}.
It suffices to show that $\Lambda_j<g_0\,/\,3$  for all $j$.
 
Since $(t,x) \in \mathcal{D}_{\delta_2}$\,, we can take advantage of the property \eqref{eq:H} and get $\Lambda_1<g_0\,/\,3\,.$ 
Moreover, we notice that ${\mathcal{L}}_{a}\varphi(t,x)=\sum_{i,j}  a_{ij}H_{ij}(t,x) $ (from the definitions \eqref{eq:hatL} of $\mathcal{L}_{a}$ and \eqref{eq:H_{ij}} of $H_{ij}$) and obtain
\begin{equation*}
\Lambda_2
= \left|\sum_{i,j}  a_{ij}\big[H_{ij}(t_0,x_0) -H_{ij}(t,x) \big] \right|
< n^2\cdot C \cdot g_0\,/\,(3n^2C)
=g_0\,/\,3\quad {\mathrm{by}}\  \eqref{eq:H}\,,
\end{equation*}
\begin{equation*}
\Lambda_3
= \left|\sum_{i,j}  \big[(a_0)_{ij}-a_{ij}\big] H_{ij}(t_0,x_0) \right|
< n^2\cdot\iota \cdot \left( \max_{i,j} |H_{ij}(t_0,x_0)| \right) 
< g_0\,/\,3\quad {\mathrm{by}}\  \eqref{eq:iota}\,.
\end{equation*}
This completes the proof.
\end{proof}

\medskip

Take  $x=x_0$ and $a=a_0$ in Assumption \ref{asmp2} with $\iota$ defined in \eqref{eq:iota}. 
Let $\delta$, $\mathbf{s}$, $\boldsymbol{\theta}$ and  $\mathbf{a}$ be the corresponding elements described in Assumption \ref{asmp2}.
We shall now adopt the definitions of $C_1$  from \eqref{eq:C_1} and  introduce the strictly positive constants  
\begin{equation}\label{eq:C^*_3}
C^*_2:= \min_{\partial \mathcal{D}_{\delta}} \, ({\Phi}_* - \varphi )(t,x) > 0\quad (\mathrm{by}\ \eqref{eq:min})\,, \quad C^*_3:=\frac{C^*_2\, e^{-C_1}(||x_0||_1-n\delta)}{2(||x_0||_1 + n\delta)}> 0  \quad (\mathrm{by}\ \eqref{eq:||x||_1})  
 \end{equation}
by analogy with $C_2$ and $C_3$ in \eqref{eq:C_1}.
We observe from the definition \eqref{eq:u_*} of ${\Phi}_*$ that 
$$
\liminf_{(t,x)\to (t_0,x_0)}({\Phi}-\varphi)(t,x) = ({\Phi}_*-\varphi)(t_0,x_0)=0\,,
$$ 
hence there exists $(t^*,x^*)\in \mathcal{D}_{\delta}$ such that 
\begin{equation}\label{eq:t*,x*2}
({\Phi}-\varphi)(t^*,x^*)<C^*_3\,;
\end{equation}
and thanks to Assumption \ref{asmp2}, 
there exists an admissible system $\mathcal{M}^{x^*}\in \mathfrak{M}(x^*)$ with the functionals $\sigma$ and $\vartheta$ defined by \eqref{eq:Markovian}.
The remaining discussion in this section (with the exception of  Proposition \ref{prop:DPP}) will be carried out under this admissible system.

\smallskip

Now we shall adopt the definitions of  $\nu$, $\lambda$ and $\rho$ from \eqref{eq:nu} and  \eqref{eq:rho}. 
For any $0\le s\le \rho$\,, 
we have  $\big(t^*-s,\mathfrak{X}(s)\big) \in \overline{\mathcal{D}_{\delta}}\subset \mathcal{D}_{\delta_2}$ and therefore \eqref{eq:iota}  holds  for $(a,t,x)=\left( \alpha(s,\mathfrak{X}),t^*-s,\mathfrak{X}(s)\right) $ by virtue of \eqref{eq:continuous a} (recall from \eqref{eq:Markovian} that $\alpha(s,\mathfrak{X})=\mathbf{a}(\mathfrak{X}(s))$). 
Therefore, we can apply \eqref{eq:G<0} and obtain
\begin{equation}\label{eq:G(s)<0}
\mathcal{G}_{\alpha(s,\mathfrak{X})}\big(t^*-s,\mathfrak{X}(s)\big)<0\,.
\end{equation}

\medskip
Let us apply now Lemma \ref{lemma:Ito} with $T=t^*$, integrating \eqref{eq:Ito} with respect to $t$ over $[0, \rho]$  and taking the expectation under $\mathbb{P}$, to obtain
\begin{equation}\label{eq:E[LXphi2]}\ 
||x^*||_1\,\varphi(t^*,x^*) 
- \mathbb{E} \left[L(\rho)X(\rho)\varphi \big(t^*-\rho,\mathfrak{X}(\rho)\big)\right] 
=\, \mathbb{E} \left[\int_0^{\rho} L(s)X(s)\, g\big(t^*-s,s,\mathfrak{X}\big)\, \mathrm{d}s \right]  <\, 0\,,
\end{equation}
by \eqref{eq:G(s)<0} and the same reasoning as right below \eqref{eq:E[LXphi]}. Here $g$ is defined in \eqref{eq:hatG}, and thus the quantity $g\big(t^*-s,s,\mathfrak{X}\big)$ is  the left-hand side of \eqref{eq:G(s)<0} with $\ell=\ell^*$.

\smallskip
Notice that  $\,\big(t^*-\nu,\mathfrak{X}(\nu)\big) \in \partial\mathcal{D}_{\delta}\,$ holds
by the definition \eqref{eq:nu} of $\nu$, thus  
\begin{equation}\label{eq:partialD2}
\varphi \big(t^*-\nu,\mathfrak{X}(\nu)\big)
\le {\Phi}_* \big(t^*-\nu,\mathfrak{X}(\nu)\big) - C^*_2\,.
\end{equation}
Plugging \eqref{eq:min}, (\ref{eq:t*,x*2}) and (\ref{eq:partialD2}) into (\ref{eq:E[LXphi2]}) yields
\begin{eqnarray}
\nonumber
&&  0> ||x^*||_1\big[ -C^*_3 + {\Phi}(t^*,x^*) \big]
- \mathbb{E} \left[\bm{1}_{\{\rho=\nu\}}L(\rho)X(\rho)\left({\Phi}_* \big(t^*-\rho,\mathfrak{X}(\rho)\big)-C^*_2\right)\right. \\ \label{eq:<0}
&&\ \ \ \ \ + \left.\bm{1}_{\{\rho\neq\nu\}}L(\rho)X(\rho){\Phi}_* \big(t^*-\rho,\mathfrak{X}(\rho)\big)\right] \\ \nonumber
&&     = 
-C^*_3\, ||x^*||_1+ ||x^*||_1\, {\Phi}(t^*,x^*) - \mathbb{E} \left[L(\rho)X(\rho){\Phi}_* \big(t^*-\rho,\mathfrak{X}(\rho)\big)\right]
+ C^*_2\, \mathbb{E} \big[\bm{1}_{\{\rho=\nu\}}L(\rho)X(\rho)\big].
\end{eqnarray}
On the strength of (\ref{eq:L(rho)}), Lemma \ref{lemma:rho=nu}, the  definition (\ref{eq:C^*_3}) of $C^*_3$, and \eqref{eq:||x||_1}, we obtain now
\begin{eqnarray*}
C^*_2\, \mathbb{E} \big[\bm{1}_{\{\rho=\nu\}}L(\rho)X(\rho)\big]   &\ge& 
  C^*_2\, e^{-C_1}(||x_0||_1 - n\delta)\,  \mathbb{P}\big(\rho=\nu \big)\\
&\ge& C^*_2\, e^{-C_1}(||x_0||_1 - n\delta)\, /\, 2 
= C^*_3(||x_0||_1 + n\delta) > C^*_3\, ||x^*||_1  \,.
\end{eqnarray*}
(This explains why we constructed $C^*_3$ as we did in \eqref{eq:C^*_3}; in fact, setting $C^*_3$ to be any value less than the right-hand side of \eqref{eq:C^*_3} would also work.) 

Substituting this inequality into (\ref{eq:<0}), and recalling  ${\Phi}_*(\cdot,\cdot)\le {\Phi}(\cdot,\cdot)$ from the definition \eqref{eq:u_*} of ${\Phi}_*\,$,  leads now to the inequality
 \begin{equation*}  
||x^*||_1 \, {\Phi}(t^*,x^*)< \mathbb{E} \left[L(\rho)X(\rho)\, {\Phi}\big(t^*-\rho,\mathfrak{X}(\rho)\big)\right] 
\end{equation*} 
of \eqref{eq:xPhi<E[LXPhi]}. However, this inequality contradicts the Dynamic Programming Principle of Proposition \ref{prop:DPP}  right below, so the proof of Theorem \ref{thm:viscosity2} is complete.  
\end{proof}

\begin{proposition}\label{prop:DPP}
{\bf Dynamic Programming Principle} {\rm (\cite{NN13}--\cite{NV})}{\bf :}
For any given $\, (T,x) \in (0, \infty) \times \mathbb{R}^n_+$ and any stopping time $\, \tau\leq T<\infty \,$, we have
\begin{equation*} 
||x||_1 \, {\Phi}(T,x) = \sup_{\mathcal{M}\in{\mathfrak{M}}(x)}
\mathbb{E}^{\mathbb{P}^{\mathcal{M}}}\left[L^{\mathcal{M}}(\tau)X^{\mathcal{M}}(\tau) \, {\Phi}\big(T-\tau,\mathfrak{X}^{\mathcal{M}}(\tau)\big)\right].
\end{equation*}
\end{proposition}
\begin{proof}
We refer to \cite[Proposition 2.2, Theorem 2.4, Remark 2.7]{NN13} and \cite[Theorem 2.3]{NV}.
\end{proof}

\section{Viscosity Characterization of the Arbitrage Function}\label{section:u}

Let us go back to the arbitrage function $\mathfrak{u}$ of \eqref{eq:u}. As a consequence of the minimality result
Theorem \ref{thm:min} below,  if $\Phi$ of \eqref{eq:Phi} is a classical supersolution of \eqref{eq:PDE}, then the  function $\mathfrak{u}$ coincides with  $\Phi$ and hence is the smallest nonnegative classical supersolution of the Cauchy problem of \eqref{eq:PDE}, \eqref{eq:initial cond}; in fact, we have $\mathfrak{u}\equiv \Phi$ if  $\Phi$ is only continuous (see Theorem \ref{thm:u=Phi} below). 

\begin{theorem}\label{thm:min}{\em{ (\eqref{eq:u>=Phi} and \cite[Proposition 2]{FK11}) }}
For any nonnegative classical supersolution $U$ of the C{\scriptsize AUCHY} problem \eqref{eq:PDE}, \eqref{eq:initial cond}, we have
$$U(T,x)\ge \mathfrak{u}(T,x)\ge \Phi(T,x)\ge \widehat{\Phi}(T,x)>0\,, \ \  \forall\ (T,x)\in[0, \infty)\times \mathbb{R}_{+}^n\,.$$
\end{theorem}
\begin{proof}
We adopt the idea from the proof in \cite[Proposition 2,\ (5.3)--(5.15)]{FK11}; the detailed proof is provided   in section \ref{appendix:min}. 
\end{proof}

\begin{theorem}\label{thm:u=Phi}
The arbitrage function $\mathfrak{u}$ coincides with the function $\,\Phi$ of \eqref{eq:Phi} if $\,\Phi$ is continuous.
\end{theorem}

This theorem is proved right below. Combining it with Theorems \ref{thm:viscosity1}, \ref{thm:viscosity2} and \ref{thm:min}, and recalling Remark \ref{rmk:u*} and $\,\widehat{\Phi}^*$ from \eqref{eq:u^*},   gives the following characterizations of the arbitrage function $\mathfrak{u}$\,.

\begin{corollary}\label{coro:u sol}
Suppose that the conditions in Theorem \ref{thm:viscosity2} are in force and the function $\,\Phi$ is continuous. 

Then the arbitrage function $\mathfrak{u}$ is a viscosity supersolution of the HJB equation \eqref{eq:PDE} subject to
the initial condition \eqref{eq:initial cond}.
If furthermore $\,\Phi\equiv \widehat{\Phi}^*$, then $\mathfrak{u}$ is a viscosity solution of \eqref{eq:PDE} subject to \eqref{eq:initial cond}.

\smallskip

If in addition $\mathfrak{u}$ is of class $C\left([0, \infty)\times \mathbb{R}_{+}^n\right)\, \cap\, C^{1,2}\left((0, \infty)\times \mathbb{R}_{+}^n\right)$, 
then it is the smallest nonnegative classical (super)solution of the C{\scriptsize AUCHY} problem \eqref{eq:PDE}, \eqref{eq:initial cond}.
\end{corollary}

\begin{remark}
{\rm
If a robust strong arbitrage relative to the market exists on some time horizon $[0,T ]$ for some initial capitalization $x$ (see Remark \ref{rmk:u<1}), then $\mathfrak{u}(T,x)<1$. This amounts to a failure of uniqueness of classical/viscosity solutions for the \textsc{Cauchy} problem of \eqref{eq:PDE}, \eqref{eq:initial cond}, since the constant  $u\equiv 1$ is always a (trivial) solution to this problem.

We refer the reader to \cite[p.\,2205]{FK11} or to \cite{Lyons}, for an interpretation of Theorem \ref{thm:u=Phi}. 
} \qed
\end{remark}

\begin{proof}[Proof of Theorem \ref{thm:u=Phi}:] 
Let \,$\mathcal{U}$\label{p:mathcalU} be the collection of positive classical supersolutions of the Cauchy problem \eqref{eq:PDE}, \eqref{eq:initial cond}, and \,$\breve{\mathcal{U}}$\label{p:breve{mathcalU}} the collection of continuous functions \label{p:breveU}$\breve{U}:[0,\infty)\times \mathbb{R}^n_+\to \mathbb{R}_+$ that satisfies \eqref{eq:initial cond}  and that the process $L(t)X(t)\Phi(T-t,\mathfrak{X}(t))$ is
a supermartingale under every admissible system.
Note that $\Phi\in \breve{\mathcal{U}}$ by virtue of  \cite[Theorem 2.3]{NV}. 

Following the idea in \cite[Theorem 1]{FK11}, we have  for $T=0$\, the identities $\mathfrak{u}(0, x)=1=\Phi(0, x)$ for all $x \in \mathbb{R}^n_+$ by the initial condition \eqref{eq:initial cond}. Now we fix an arbitrary pair $(T, x) \in (0,\infty)\times \mathbb{R}^n_+$\,. For every $\varepsilon >0$\,, there exists a mollification \label{p:U_epsilon}$U_{\varepsilon}\in \mathcal{U}$ of the function $\Phi$ with $0<U_{\varepsilon}(T , x) \le \Phi(T , x)+\varepsilon$.  
Combining with Theorem \ref{thm:min} gives
$$\mathfrak{u}(T , x) \le U_{\varepsilon}(T , x)\le \Phi(T, x) + \varepsilon\,.$$
Since $\varepsilon >0$ is arbitrary, this
leads to $\mathfrak{u}(T, x)\le \Phi(T, x)$. On the other hand, the reverse inequality $\mathfrak{u}(T , x)\ge \Phi(T , x)$ holds on the
strength of \eqref{eq:u>=Phi}. Hence, $\mathfrak{u}(T, x)=\Phi(T, x)$ on $[0,\infty)\times \mathbb{R}^n_+$\,.
\end{proof}

\begin{remark}
{\rm 
With slight modifications  our approach can also show that, under appropriate conditions described in Theorems \ref{thm:viscosity1}, \ref{thm:viscosity2} and Corollary \ref{coro:u sol} but now with 
\begin{equation*} 
F(t,x,r,p,q) \,=\, - \, \frac{1}{\,2\,}\sup_{a\in \mathcal{A}(x)} \left( \sum_{i,j}  x_i x_j a_{ij} q_{ij} \right), \ \ \mathrm{for}\ \  q=(q_{ij})_{1\le i,j\le n}\,,\ \ p=(p_1,\dots,p_n)',
\end{equation*}
and $a=(a_{ij})_{1\le i,j\le n}$\,, the functions $${\Psi}(T,x)=||x||_1\,{\Phi}(T,x)\,,\ \  \qquad \widehat{\Psi}(T,x)=||x||_1\,\widehat{\Phi}(T,x)$$ and $$ 
\mathfrak{v}(T,x) := ||x||_1 \,\mathfrak{u}(T,x)\,, \ \ \ \ (T,x)\in [0,\infty)\times \mathbb{R}^n_+$$  
are classical/viscosity (super/sub)solutions of an HJB equation simpler than \eqref{eq:PDE} -- namely, the {\it P{\scriptsize UCCI}-maximal type equation} 
\begin{equation} \label{eq:Pucci}
u_t (t,x) -\frac{1}{\,2\,}\sup_{a\in \mathcal{A}(x)} \left( \sum_{i,j}  x_i x_j a_{ij} D^2_{ij} u    (t,x)  \right)  = 0\,,\ \ \ \ (t,x) \in (0, \infty)\times \mathbb{R}_{+}^n 
\end{equation} 
subject to the initial condition $u(0,x)=||x||_1$\,,   
and   are dominated by any nonnegative classical supersolution of the \textsc{Cauchy} problem \eqref{eq:Pucci}, \eqref{eq:initial cond}. }  \qed
\end{remark} 

\subsection{Sufficient Conditions for $\mathfrak{u}\equiv\Phi\equiv \widehat{\Phi}$ to be a classical supersolution of (\ref{eq:PDE})}

Now let us provide some sufficient conditions under which we have $\mathfrak{u}\equiv\Phi\equiv \widehat{\Phi}$, and this function is a classical solution of \eqref{eq:PDE} -- thus also the smallest nonnegative classical (super)solution of the \textsc{Cauchy} problem \eqref{eq:PDE}, \eqref{eq:initial cond} by virtue of Theorem \ref{thm:min}. 

In particular, via the discussions below, we will see that one sufficient condition is the following specific requirements on the  Knightian  uncertainty $\mathbb{K}$\,.
 
\begin{proposition}\label{prop:suff}
Suppose that there exist locally L{\scriptsize IPSCHITZ} functions $\mathbf{s}: \mathbb{R}_+^n\to {\mathrm{GL}}(n)$ and  $\boldsymbol{\theta}: \mathbb{R}^n_+\to \mathbb{R}^n$, and subsets $\mathcal{R}(y)$ $(y\in \mathbb{R}_+^n)$ of $\,\mathbb{R}$  such that 
the functions  $\mathbf{s}(\cdot)$ and $\mathbf{b}(\cdot):=\mathbf{s}(\cdot)\boldsymbol{\theta}(\cdot)$ are linearly growing $($i.e., satisfy  \eqref{eq:linearGrowth}$)$  and   with $
\mathbf{a}(y) :=\mathbf{s}(y)\mathbf{s}'(y)$ $(y\in \mathbb{R}^n_+)$  we have
\begin{equation}\label{eq:(theta,a)}
\big(\boldsymbol{\theta}(y),\mathbf{a}(y)\big)\in \mathcal{K}(y)\,,\quad \mathcal{A}(y)=\{r\cdot\mathbf{a}(y): r\in \mathcal{R}(y)\}\,,\quad \min\, \mathcal{R}(y)=1 \quad {\mathrm{for\ all\ }}  y\in \mathbb{R}^n_+\,.
\end{equation}
Then $\mathfrak{u}\equiv \Phi\equiv \widehat{\Phi}$ 
is the smallest nonnegative classical (super)solution of the C{\scriptsize AUCHY} problem \eqref{eq:PDE}, \eqref{eq:initial cond}, and the smallest nonnegative classical (super)solution of the C{\scriptsize AUCHY} problem \eqref{eq:a},  \eqref{eq:initial cond}.
\end{proposition}
\begin{proof}
This result follows directly from Remark \ref{rmk:Markovian} and  Theorem \ref{prop:u_M supersol} below.
\end{proof}

We start with the following observation.

\begin{proposition}\label{prop:Phi supersol}
If there exist admissible systems $\,\mathcal{M}^{y}\in \widehat{\mathfrak{M}}(y)$ $(y\in \mathbb{R}_{+}^n)$ such that \begin{equation}\label{eq:V}
V(t,y):=\mathfrak{u}_{\mathcal{M}^{y}}(t,y)\,, \quad (t,y)\in [0, \infty)\times\mathbb{R}_{+}^n
\end{equation} 
is a classical supersolution of \eqref{eq:PDE}, then  $\mathfrak{u}\equiv\Phi\equiv\widehat{\Phi}\equiv V$ is the smallest nonnegative classical (super)solution of the C{\scriptsize AUCHY} problem \eqref{eq:PDE}, \eqref{eq:initial cond}.
\end{proposition} 
\begin{proof}
Theorem \ref{thm:min} and the definition \eqref{eq:Phi} of $\widehat{\Phi}$ give $V(t,y)\ge \mathfrak{u}(t,y) \ge \Phi(t,y)\ge\widehat{\Phi}(t,y)\ge \mathfrak{u}_{\mathcal{M}^{y}}(t,y)=V(t,y)$,   hence $\mathfrak{u}\equiv\Phi\equiv\widehat{\Phi}\equiv V$ is a classical supersolution of \eqref{eq:PDE}. 
\end{proof}

To proceed further, we need the following assumption. 

\begin{asmp}\label{asmp:Markovian}
{\rm 
There exist admissible systems $\mathcal{M}^x \in \widehat{\mathfrak{M}}(x)$, $x\in \mathbb{R}_{+}^n$ 
such that \\
{\bf (i)} they share the same functionals $\sigma(t,\mathfrak{X})=\mathbf{s}(\mathfrak{X}(t))$ and $ \vartheta(t,\mathfrak{X})=\boldsymbol{\theta}(\mathfrak{X}(t))$  as in  \eqref{eq:Markovian}; and\\
{\bf (ii)} for every $x\in \mathbb{R}_{+}^n$\,,
the process $\mathfrak{X}$ in $\mathcal{M}^x$ is unique in distribution in the following sense (and thus strongly Markovian): for any admissible system $\widetilde{\mathcal{M}} 
\in {\mathfrak{M}}(x)$ with the same functionals $\sigma$ and $\vartheta$  as in $\mathcal{M}^x$, the two processes $\mathfrak{X}^{\mathcal{M}^x}$ and $\mathfrak{X}^{\widetilde{\mathcal{M}} }$ have the same law. 
}\qed
\end{asmp}

\begin{remark}\label{rmk:Markovian}
{\rm 
Assumption \ref{asmp:Markovian} holds when the conditions in Remark \ref{rmk:nonemptyM}\,(i) are satisfied.
} \qed
\end{remark}

\begin{proposition}\label{prop:u_M vis sol} 
Under Assumption \ref{asmp:Markovian}, the function $V$ of \eqref{eq:V} is\\
{\bf (i)} dominated by any nonnegative classical supersolution of the C{\scriptsize AUCHY} problem \eqref{eq:a},  \eqref{eq:initial cond};\\
{\bf (ii)} a viscosity solution of \eqref{eq:a}, if $ \boldsymbol{\theta}(\cdot)$ and $\mathbf{a}(\cdot)$ are locally bounded and $\mathbf{a}(\cdot)$ is continuous; and\\
{\bf (iii)} a classical solution of the C{\scriptsize AUCHY} problem \eqref{eq:a}, \eqref{eq:initial cond} (and thus its smallest nonnegative (super)solution), if $\mathbf{s}(\cdot)$ and $\boldsymbol{\theta}(\cdot)$ are locally L{\scriptsize IPSCHITZ}.
\end{proposition}
\begin{proof}
We will see that {\it (i)} and {\it (ii)} are  special cases of Theorems \ref{thm:min} and Theorems \ref{thm:viscosity1}--\ref{thm:viscosity2}, respectively,  with $\mathcal{K}(y)=\{(\boldsymbol{\theta}(y),\mathbf{a}(y))\}$ $(y\in \mathbb{R}_{+}^n)$ via the following observations.
First, in this case we have $\widehat{\mathcal{L}}(t,y)=\mathcal{L}_{\mathbf{a}(y)}(t,y)$ (recall the definition \eqref{eq:hatL} for $\widehat{\mathcal{L}}$ and $\mathcal{L}_{a}$). 
Moreover, by virtue of Assumption \ref{asmp:Markovian} and definition \eqref{eq:um}, we have $\mathfrak{u}_{\mathcal{M}^{y}}(t,y)=\mathfrak{u}_{\mathcal{M}}(t,y)$ for all $\mathcal{M}\in {\mathfrak{M}}(y)$, and by the definition \eqref{eq:Phi} of $\Phi$ and $\widehat{\Phi}$ gives  $$\Phi(t,y)=\widehat{\Phi}(t,y)=\mathfrak{u}_{\mathcal{M}^{y}}(t,y)=V(t,y)\, .$$  
{\it (iii)} Under these conditions, we have $V(\cdot \,,\cdot)\in C^{1,2}((0,\infty)\times \mathbb{R}_+^n)$ 
(see \cite[Theorem 4.7]{R13} for a proof that uses results from the theory of stochastic flows (\cite{Ku}, \cite{Pr}) and from parabolic partial differential equations (\cite{ET}, \cite{JT})), and conclude by invoking (ii) since the local \textsc{Lipschitz} condition on $\mathbf{s}$ and $\boldsymbol{\theta}$ implies the condition in {\it (ii)}. 
\end{proof}

\begin{proposition}\label{prop:u_M supersol} 
If Assumption \ref{asmp:Markovian} holds with locally L{\scriptsize IPSCHITZ} functions $\mathbf{s}(\cdot)$ and $\boldsymbol{\theta}(\cdot)$, and there exist subsets $\mathcal{R}(y)$ $(y\in \mathbb{R}_+^n)$ of $\,\mathbb{R}$  such that \eqref{eq:(theta,a)} holds, then, with  $\mathcal{M}^{y}\in \widehat{\mathfrak{M}}(y)$ as in Assumption \ref{asmp:Markovian}, the function
\begin{equation}\label{eq:u=V}
\mathfrak{u}(t,y) \, \equiv \, \Phi(t,y)\, \equiv \, \widehat{\Phi}(t,y) \, \equiv \, \mathfrak{u}_{\mathcal{M}^y}(t,y)
\end{equation} 
 is the smallest nonnegative classical (super)solution of the C{\scriptsize AUCHY} problem \eqref{eq:PDE}, \eqref{eq:initial cond}, as well as the smallest nonnegative classical (super)solution of the C{\scriptsize AUCHY} problem \eqref{eq:a}, \eqref{eq:initial cond}. 
\end{proposition}
\begin{proof}
By Proposition \ref{prop:u_M vis sol}\,(iii), the right-hand side of \eqref{eq:u=V}, i.e., the function $V$ of \eqref{eq:V}  solves \eqref{eq:a}:
\begin{equation*}
V_t(t,y)=\mathcal{L}_{\mathbf{a}(y)}V(t,y)\,, \quad (t,y)\in (0, \infty)\times\mathbb{R}_{+}^n\,.
\end{equation*}
Thus
\begin{equation*}
\mathcal{L}_{r\cdot\mathbf{a}(y)}V(t,y)\,=\,r\cdot\mathcal{L}_{\mathbf{a}(y)}V(t,y)\,=\,r\cdot V_t(t,y)\,, \quad (t,y)\in (0, \infty)\times\mathbb{R}_{+}^n\,.
\end{equation*}
Once we have shown that $V_t(t,y)\le 0$\,, i.e., that $V$ is nonincreasing in $t$ on $(0, \infty)$, for all $y\in \mathbb{R}_{+}^n$\,, then $V$ is a classical supersolution of \eqref{eq:PDE} on the strength of  \eqref{eq:(theta,a)}, and the proof will be complete by  Proposition  \ref{prop:Phi supersol}. 
In fact, under any given admissible system, the positive process $L(\cdot)X(\cdot)$ is a local martingale, hence a supermartingale (one can derive the formula $\mathrm{d}( L(t)X(t) ) =L(t)X(t)\left( \pi'\sigma-\vartheta'\right) (t,\mathfrak{X})\, \mathrm{d}W(t) $ with $\pi$ the market portfolio, via \textsc{It\^o}'s Rule; see \eqref{eq:dLZ} for details). Therefore
$V(t,y)=\mathbb{E}^{\mathbb{P}^{\mathcal{M}^y}}\big[L^{\mathcal{M}^y}(t)X^{\mathcal{M}^y}(t)\big]/\,||y||_1$ 
is indeed nonincreasing in $t$. 
\end{proof}

\begin{remark}
{\rm
This result is in agreement with general regularity
theory for fully nonlinear parabolic equations, as in \cite[Theorem II.4]{L83c}.
}
\end{remark}

\begin{remark}
{\rm
We have tried to find weaker conditions for Theorem \ref{thm:u=Phi} to hold, or for the function $\Phi$ to be continuous, but did not succeed. Even if all the functions $\mathfrak{u}_{\mathcal{M}}$   are of class $C^{1,2}$, their supremum $\Phi$ might still  fail to be continuous.}  \qed
\end{remark}

\section{The Proof of Theorem \ref{thm:min}: Minimality}
 \label{appendix:min}

The proof consists of two parts, Theorems \ref{thm:U>=Phi} and \ref{thm:U>=u}. Theorem \ref{thm:U>=Phi} shows that any nonnegative classical supersolution $U$ of the C{\scriptsize AUCHY} problem \eqref{eq:PDE}, \eqref{eq:initial cond} is strictly positive, by proving that
$U(T,x)\ge \Phi(T,x)$ and then applying the fact $\Phi(T,x)>0$ from \eqref{eq:u>=Phi}.

In Theorem \ref{thm:U>=u}, the positivity of $U$ from Theorem \ref{thm:U>=Phi} enables us to construct an investment rule from $U$ (see \eqref{eq:pi^U} below) that matches or outperforms the market portfolio over the time horizon $[0,T]$, with probability one under all admissible systems. We then conclude that $U(T,x)\ge \mathfrak{u}(T,x)$ from  the definition \eqref{eq:u} of $\mathfrak{u}(T,x)$.

The following proofs of Theorems \ref{thm:U>=Phi} and \ref{thm:U>=u} adopt the idea from \cite[Proposition 2, (5.3)--(5.15)]{FK11} and provide details for completeness. 

\begin{theorem}\label{thm:U>=Phi}
For any nonnegative classical supersolution $U$ of the C{\scriptsize AUCHY} problem \eqref{eq:PDE}, \eqref{eq:initial cond}, we have
\begin{equation}\label{eq:U>=Phi}
U(T,x)\ge \Phi(T,x)>0\,, \ \  \forall\ (T,x)\in[0, \infty)\times \mathbb{R}_{+}^n\,.
\end{equation}
\end{theorem}
\begin{proof}
The second inequality was shown in \eqref{eq:u>=Phi}. 
For the first inequality, let us fix an admissible system $\mathcal{M} \in \mathfrak{M}(x)$; the remaining discussion in this proof will be carried out under this system. The key point, is to show that the process 
\begin{equation} \label{eq:Xi}
\Xi(t):=X(t)L(t)U\big(T-t,\mathfrak{X}(t)\big)
\end{equation}
is a supermartingale. 
Once this is proved, with the initial condition $U(0,\cdot)\ge 1$\,, we obtain
\begin{align*}
 ||x||_1\, U(T,x)= &\ 
 \mathbb{E}\big[ \Xi(0)\big] 
\ge \mathbb{E}\big[ \Xi(T)\big] 
=\mathbb{E}\big[X(T)L(T)U(0,\mathfrak{X}(T))\big]   \\ 
 \ge &\   \mathbb{E}\big[X(T)L(T)\big]   
= 
||x||_1\, \mathfrak{u}_{\mathcal{M}}(T,x)\,,\quad \mathrm{by\ the\ definition}\ \eqref{eq:um}\,.
\end{align*}
Since $||x||_1>0$\,, we deduce $\, U(T,x)\ge \mathfrak{u}_{\mathcal{M}}(T,x)\,$, 
which leads to \eqref{eq:U>=Phi} by the definition \eqref{eq:Phi}.

\smallskip
To show the supermartingale property of $\Xi(\cdot)$, we apply Lemma \ref{lemma:Ito} with $\varphi=U$ and get
\begin{eqnarray}\label{eq:dXi}
\nonumber   \mathrm{d}\left(\Xi(t) \right) 
&=& -L(t)X(t)\left(U_t- \mathcal{L}_{\alpha(t,\mathfrak{X})}U\right) \big(T-t,\mathfrak{X}(t)\big)\, \mathrm{d}t 
-X(t)U\big(T-t,\mathfrak{X}(t)\big)L(t)\,\vartheta'(t,\mathfrak{X})\,  \mathrm{d}W(t)\\
&&+L(t)\sum_{i,k} X_i(t)  \left[U\big(T-t,\mathfrak{X}(t)\big)+X(t) D_i U \big(T-t,\mathfrak{X}(t)\big)\right] \sigma_{ik}(t,\mathfrak{X})\, \mathrm{d}W_k(t)\,.
\end{eqnarray}
Thanks to the supersolution property of $U$ that
\begin{equation} \label{eq:Ut-LU}
\left( U_t- \mathcal{L}_{\alpha(t,\mathfrak{X})}U\right) (s,y)
\ge  \big( U_t- \widehat{\mathcal{L}}U\big) (s,y) \ge 0\,,\quad \forall\  (s,y)\in[0,\infty)\times\mathbb{R}^n_+
\end{equation} 
(recall $\mathcal{L}_{a}$ and $\widehat{\mathcal{L}}$ from \eqref{eq:hatL}) and the nonnegativity of the processes $L(\cdot) $, $X(\cdot)$ and the function $U(\cdot,\cdot)$, we conclude that $\Xi(t)$ is a nonnegative local martingale, hence a supermartingale.
\end{proof}

\begin{theorem}\label{thm:U>=u}
For any nonnegative classical supersolution $U$ of the C{\scriptsize AUCHY} problem \eqref{eq:PDE}, \eqref{eq:initial cond}, the investment rule $\pi^U\in \mathfrak{P}$ generated by this function $U$ through
\begin{equation}\label{eq:pi^U}
\pi_i^U(t,\omega):=\omega_i(t)D_i \log U(T-t,\omega(t))+\frac{\omega_i(t)}{||\omega(t)||_1}\,,\,\,\,\,  i=1,\dots,n\,,\,\,  \,  t\in[0, T]
\end{equation}
for continuous function $\omega:[0,\infty)\to \mathbb{R}_{+}^n\,$,  
satisfies the inequality
\begin{equation}\label{eq:Z>=X}
Z^{\,U(T,x)X^{\mathcal{M}}(0),\pi^U}(T)\ge X^{\mathcal{M}}(T)\,,\quad \mathbb{P}\mathrm{-a.s.}, \quad \forall\ \mathcal{M} \in \mathfrak{M}(x)\,.
\end{equation}
It then follows from the definition \eqref{eq:u} of $\,\mathfrak{u}(T,x)$ that
$$U(T,x)\ge \mathfrak{u}(T,x)\,, \ \  \forall\ (T,x)\in[0, \infty)\times \mathbb{R}_{+}^n\,.$$
\end{theorem}

\begin{proof}
The investment rule $\pi^U$ is well-defined since $U$ is positive by Theorem \ref{thm:U>=Phi}. Let us fix $\mathcal{M} \in \mathfrak{M}(x)$; the remaining discussion in this proof will be carried out under this system. 

We shall set $v:=U(T,x)X(0)$ and ${\pi}:=\pi^U$\label{p:pi}. 
The main goal is to show that the growth rate of the process  $\log \left(  L(t)Z^{v,\pi}(t)\right)  $ is no less than that of $\log \Xi(t)$ with $\Xi(t)$ defined in \eqref{eq:Xi}.
Once this is proved, noticing that these two processes start at the same initial value $v$, we obtain
\begin{equation*}
L(T)Z^{v,\pi}(T)\ge \Xi(T)=X(T)L(T)U(0,\mathfrak{X}(T))\ge X(T)L(T)\,,
\end{equation*}
as $U(0,\cdot)\ge 1$ by the  initial condition. This leads to \eqref{eq:Z>=X} as $L(T)>0$\,.

\smallskip
To start, we observe from \eqref{eq:Z} with $\pi={\it \Pi}$ that the wealth process $Z^{v,\pi}(\cdot)$ satisfies the dynamics
\begin{equation}\label{eq:dZ}
\mathrm{d}Z^{v,\pi}(t)
= Z^{v,\pi}(t) \pi'(t,\mathfrak{X})\sigma(t,\mathfrak{X}) \left[\vartheta(t,\mathfrak{X})\, \mathrm{d} t+ \mathrm{d}W(t) \right] \quad \mathrm{with }\quad  Z^{v,\pi}(0)=v\,.
\end{equation}

We apply \textsc{It\^o}'s Rule
for the product function $f_1(r_1,r_2):=r_1r_2$ with \eqref{eq:L} and \eqref{eq:dZ} yields
\begin{align}\label{eq:dLZ}
\mathrm{d}\left( L(t)Z^{v,\pi}(t) \right) 
=&\ L(t)\, \mathrm{d} Z^{v,\pi}(t)+ Z^{v,\pi}(t)\, \mathrm{d}L(t) + \mathrm{d} \langle L, Z^{v,\pi}\rangle(t)  \\ \nonumber
=&\ L(t)Z^{v,\pi}(t) 
\left[ \pi'\sigma\vartheta\,\mathrm{d}t + \pi'\sigma\,\mathrm{d}W(t) - \vartheta'\,\mathrm{d}W(t) 
- \pi'\sigma\vartheta\,\mathrm{d}t \right](t,\mathfrak{X})  \\ \nonumber
=&\  L(t)Z^{v,\pi}(t)\mathcal{H}(t,\mathfrak{X})\, \mathrm{d}W(t) \,,
\end{align}
where
\begin{equation} \label{eq:calH}
\mathcal{H}(t,\mathfrak{X})
:= (\pi'\sigma -\vartheta')(t,\mathfrak{X})\,,
\end{equation}
whose $k$-th component
\begin{align}\label{eq:calH_k}
\mathcal{H}_k(t,\mathfrak{X})
=\sum_{i} \left[X_i(t)D_i \log U\big(T-t,\mathfrak{X}(t)\big)+\frac{X_i(t)}{X(t)}\right]\sigma_{ik}(t,\mathfrak{X}) -\vartheta_k(t,\mathfrak{X})\,, \quad   {\mathrm{by}}\ \eqref{eq:pi^U}\,.
\end{align}

Applying \textsc{It\^o}'s Rule to the logarithm function for $L(\cdot)Z^{v,\pi}(\cdot)$, we obtain
\begin{equation}\label{eq:d{LZ}}
\mathrm{d}\log \left( L(t)Z^{v,\pi}(t)\right)
= \mathcal{H}(t,\mathfrak{X})\, \mathrm{d}W(t)-\frac{1}{\,2\,}\left( \mathcal{H}\mathcal{H}'\right) (t,\mathfrak{X})\, \mathrm{d}t\,.
\end{equation}

To determine the growth rate for $\log \Xi(\cdot)$, we recast \eqref{eq:dXi} into
\begin{equation*}
\mathrm{d}\left(\Xi(t) \right)  
=\Xi(t)\left[\mathcal{I}\big(T-t,\mathfrak{X}(t)\big)\mathrm{d}t+ \mathcal{H}(t,\mathfrak{X})\,\mathrm{d}W(t)\right],
\end{equation*}
by virtue of 
$$\left( \frac{D_i U}{U}\right) (s,y)=D_i \big( \log U(s,y)\big),\quad (s,y)\in[0,\infty)\times\mathbb{R}^n_+$$ 
and \eqref{eq:calH_k}, where
\begin{equation*} 
\mathcal{I}(s,y):=- \left(\frac{U_t- \mathcal{L}_{\alpha(t,\mathfrak{X})}U}{U}\right)(s,y)\le 0\,,\quad (s,y)\in[0,\infty)\times\mathbb{R}^n_+\,,\quad \mathrm{by}\ \eqref{eq:Ut-LU}\ \mathrm{and}\ U>0\,.
\end{equation*} 
Applying \textsc{It\^o}'s Rule again to the logarithm function for $\Xi(\cdot)$ and juxtaposing with \eqref{eq:d{LZ}} leads to
\begin{align*}
\mathrm{d}\log  \Xi(t)
= {\mathcal{I}}\big(T-t,\mathfrak{X}(t)\big)\mathrm{d}t+\mathcal{H}(t,\mathfrak{X})\, \mathrm{d}W(t)-\frac{1}{\,2\,}\left( \mathcal{H}\mathcal{H}'\right) (t,\mathfrak{X})\, \mathrm{d}t \le \mathrm{d}\log \left( L(t)Z^{v,\pi}(t)\right)\,,
\end{align*}
as desired.
\end{proof}
  
\begin{remark}
{\rm
In the special case of  a model without uncertainty,  the HJB equation \eqref{eq:PDE} reduces to a linear PDE. If additionally, the functions $\sigma$ and $\vartheta$ have the form of \eqref{eq:Markovian} and are locally \textsc{Lipschitz} continuous, then the arbitrage function $\mathfrak{u}$ is also shown to be dominated by every nonnegative and lower-semicontinuous viscosity supersolution of the \textsc{Cauchy} problem for the linear PDE \eqref{eq:PDE} and \eqref{eq:initial cond} \cite[Proposition 4.7]{BHS}, that satisfies certain convexity and continuity conditions. 

This local \textsc{Lipschitz} condition on $\sigma$ and $\vartheta$ is indispensable in the proof of  \cite{BHS}. 
It is the subject of future research, to determine whether this result still  holds with weaker assumptions and in the presence of model   uncertainty.
}  \qed
\end{remark}

\section{Examples}\label{section:example}
The volatility-stabilized model was introduced in \cite{FK05} and further generalized in \cite{Pi}, but now we add some uncertainty regarding its local volatility and relative risk structure.

\begin{example} {\bf Volatility-Stabilized Model:} 
\label{example:1}
{\rm
Take constants $c_1^*\ge c_1\ge 1/2$ and $c_2\ge 1$\,, and set 
$$\mathcal{K}(y)=\big\{ \big(\gamma_2\, \mathbf{a}(y),\gamma_1\gamma_2\,\boldsymbol{\theta}(y)\big): 
\gamma_1 \in[c_1,c_1^*],\ 
\gamma_2\in[1,c_2]\big\}\,,
$$
where
\begin{equation}\label{eq:bf{s}}
\mathbf{a}(y)=\mathbf{s}(y)\mathbf{s}'(y)\ \ {\mathrm{with}}\ \  \mathbf{s}_{ij}(y)=\bm{1}_{\{i=j\}} 
(||y||_1/y_i) ^{1/2}\,,\quad \boldsymbol{\theta}_{i}(y)=(||y||_1/y_i)^{1/2}\,,\ \ 1\le i,j\le n\,.
\end{equation} 
Then the system of Stochastic Differential Equations \eqref{eq:SDE} becomes
\begin{equation*}
\mathrm{d}X_i(t) = \gamma_1 \gamma^2_2 \, \big(X_1 (t) + \cdots + X_n (t)\big)\, \mathrm{d}t 
\qquad  \qquad  \qquad  \qquad  \qquad 
\end{equation*}
$$
\qquad  \qquad  \qquad  \qquad  \qquad  
+\, \gamma_2\sqrt{X_i(t) \big(X_1 (t) + \cdots + X_n (t)\big)\,}\, \mathrm{d}W_i(t)\,,\quad i = 1,\dots,n\,,
$$

\medskip
\noindent
or equivalently, and a bit more succinctly, 
\begin{equation*}
{\mathrm{d}}\,\log (X_i(t) )\,=\, \left( \gamma_1 -\frac{1}{2}\right)  \frac{\gamma^2_2}{\mu_i(t,\mathfrak{X}) } \, {\mathrm{d}}t  + \frac{\gamma_{2}}{ \mu_i(t,\mathfrak{X})}\, {\mathrm{d}}W_i(t), \quad i = 1,\dots,n\,,
\end{equation*}
with $\mu(t,\mathfrak{X})$ the market portfolio defined in \eqref{eq:mu}.

\medskip
\noindent
For every $x \in \mathbb{R}_{+}^n$\,, $\gamma_1 \in[c_1,c_1^*]$ and $ 
\gamma_2\in[1,c_2]$, this system of SDEs has a unique-in-distribution solution $\mathfrak{X}(\cdot)$  starting at $\mathfrak{X}(0)=x$ whose $X_i(\cdot)$'s are  time-changed versions of independent squared-\textsc{Bessel}  processes (see \cite{BP}, \cite{FK05} and \cite{G} for more details). In particular, we have $\mathfrak{X}(\cdot)\in  \mathbb{R}_{+}^n$\,.

\smallskip
Moreover, this uncertainty structure satisfies the conditions in Remark \ref{rmk:u<1} and Proposition \ref{prop:u_M supersol} with the
$\mathbf{s}$, $\boldsymbol{\theta}$ as in \eqref{eq:bf{s}} and  $\mathcal{R}(y)=[1,c_2]$. 
Hence \begin{equation}\label{eq:u=Phi<1}
\mathfrak{u}(t,y)\equiv \Phi(t,y) \equiv  \widehat{\Phi}(t,y)\equiv \mathfrak{u}_{\mathcal{M}^y}(t,y)
\begin{cases}
<1\,, &\mathrm{if}\ t>0\\ =1\,, &\mathrm{if}\ t=0
\end{cases}
\end{equation}
is the smallest nonnegative classical (super)solution of the \textsc{Cauchy} problem \eqref{eq:PDE}, \eqref{eq:initial cond}, as well as the smallest nonnegative classical (super)solution of the \textsc{Cauchy} problem \eqref{eq:a}, \eqref{eq:initial cond} (recall $\mathcal{M}^y$ from Assumption \ref{asmp:Markovian}; see \cite{G} and \cite{Pal} for a computation of the joint density of $X_1(\cdot), \dots ,X_n(\cdot)$, which leads to an explicit formula for  
$$\mathfrak{u}_{\mathcal{M}^y}(t,y) \,=\, \frac{\,\Pi_{i=1}^n\, y_i\,}{\,||y||_1\,} \, \cdot \, \mathbb{E}^{\mathbb{P}^{\mathcal{M}^y}} \left[\frac{\,  \Pi_{i=1}^n X^{\mathcal{M}^y}_i(t) \,}{\,||X^{\mathcal{M}^y}(t)||_1\,} \right]$$ and shows that this function is indeed of class $C^{1,2}$).
}
\end{example} 

\begin{example} {\bf Generalized Volatility-Stabilized Model:} \label{example:2}
{\rm
Take constants $c_i^*\ge c_i \ge 0$\,, $i=1,2,\dots,n$ and $c_{n+1}\ge 1$\,,  and set 
$$\mathcal{K}(y)=\left\{ (a,\theta): a=\gamma_{n+1}\, \mathbf{a}(y),\  
\theta_{i}=\frac{\gamma_i+\gamma^2_{n+1}}{2\gamma_{n+1}}\,\boldsymbol{\theta}_{i}(y),\ 
\gamma_i \in[c_i,c_i^*],\  \gamma_{n+1}\in [1,c_{n+1}],\ 1\le i\le n\right\}\,,
$$
where $\, \mathbf{a}(y)=\mathbf{s}(y)\mathbf{s}'(y)\, $ with 
\begin{equation}\label{eq:bf{s}2}
\   \mathbf{s}_{ij}(y)=\bm{1}_{\{i=j\}} 
\left( \frac{||y||_1}{y_i}\right)^{\kappa}G(y)\,,\ \  \ \ \ \boldsymbol{\theta}_{i}(y)=\left( \frac{||y||_1}{y_i}\right)^{\kappa} G(y)\,,\ \ \ \ 1\le i,j\le n
\end{equation} 
where $\kappa$ is a positive constant and $G:\mathbb{R}_+^n\to \mathbb{R}_+$ is a bounded and locally \textsc{Lipschitz} function (Example \ref{example:1} is a special case of this model with $\kappa=1/2$ and $G\equiv 1$).

Then the system of Stochastic Differential Equations \eqref{eq:SDE} becomes
\begin{equation*}
{\mathrm{d}}X_i(t) 
= X_i(t)   \left[  \frac{\gamma_i+\gamma^2_{n+1}}{2 \left( \mu_i(t,\mathfrak{X})\right) ^{2\kappa}}\, G^2(\mathfrak{X}(t))\, {\mathrm{d}}t + \frac{\gamma_{n+1}}{\left( \mu_i(t,\mathfrak{X})\right) ^{\kappa}}\,G(\mathfrak{X}(t))\, {\mathrm{d}}W_i(t)\right], \quad i = 1,\dots,n
\end{equation*}

\smallskip
\noindent
with $\mu(t,\mathfrak{X})$ the market portfolio defined in \eqref{eq:mu}, or equivalently,
\begin{equation*}
{\mathrm{d}}\log (X_i(t) )
= \frac{\gamma_i}{2\left( \mu_i(t,\mathfrak{X})\right) ^{2\kappa}} \, G^2(\mathfrak{X}(t))\, {\mathrm{d}}t  + \frac{\gamma_{n+1}}{\left( \mu_i(t,\mathfrak{X})\right) ^{\kappa}}\,G(\mathfrak{X}(t))\, {\mathrm{d}}W_i(t), \quad i = 1,\dots,n\,,
\end{equation*}

\smallskip
\noindent
For every $x \in \mathbb{R}_{+}^n\,$, $\gamma_i \in[c_i,c_i^*]$, $i=1,2,\dots, n$ and $\gamma_{n+1}\in [1,c_{n+1}]$,  this system of SDEs has a unique-in-distribution solution $\mathfrak{X}(\cdot)$ starting at $\mathfrak{X}(0)=x$, the components  $X_i(\cdot)$ of this solution are time-changed versions of independent squared-\textsc{Bessel} processes (see \cite{BP} and \cite[Sections 2 and 4]{Pi} for more details). In particular, we have $\mathfrak{X}(\cdot)\in  \mathbb{R}_{+}^n$\,.

This uncertainty structure also satisfies the conditions in Proposition \ref{prop:u_M supersol} with the
$\mathbf{s}$, $\boldsymbol{\theta}$ as in \eqref{eq:bf{s}2} and  $\mathcal{R}(y)=[1,c_{n+1}]$, therefore 
$$\mathfrak{u}(t,y)\equiv \Phi(t,y)\equiv \widehat{\Phi}(t,y) \equiv   \mathfrak{u}_{\mathcal{M}^y}(t,y)$$ is the smallest nonnegative classical (super)solution of the \textsc{Cauchy} problem \eqref{eq:PDE}, \eqref{eq:initial cond}, as well as the smallest nonnegative classical (super)solution of the \textsc{Cauchy} problem \eqref{eq:a}, \eqref{eq:initial cond} (recall $\mathcal{M}^y$ from Assumption \ref{asmp:Markovian}). 
If in addition $G(\cdot)$ is bounded away from zero, then the condition in Remark \ref{rmk:u<1} is satisfied as well and \eqref{eq:u=Phi<1} follows.
}
\end{example}

\smallskip

\appendix

\section{The Proof of Lemma \ref{lemma:Ito}}\label{appendix:Ito}
\begin{proof}
Let $\phi(t):=\varphi(T-t,\mathfrak{X}(t))$\label{p:phi}, 
\begin{equation*}
\mathfrak{s}_{ik}(t,\mathfrak{X}):=X_i(t) \sigma_{ik}(t,\mathfrak{X})\,,\ 
\mathfrak{s}(t,\mathfrak{X}):=(\mathfrak{s}_{ik}(t,\mathfrak{X}))_{n\times n}\, \ {\mathrm{and}}\ \ 
\mathfrak{b}(t,\mathfrak{X})=(\mathfrak{b}_1,\dots,\mathfrak{b}_n)(t,\mathfrak{X}):=(\mathfrak{s}\vartheta)(t,\mathfrak{X})\,.
\end{equation*}
Then the SDE \eqref{eq:SDE} can be rewritten as
\begin{equation}
\label{eq:SDE3}
\mathrm{d}X_i(t) = \mathfrak{b}_i(t,\mathfrak{X})\, \mathrm{d}t + \sum_{k}\mathfrak{s}_{ik}(t,\mathfrak{X}) \,\mathrm{d}W_k(t)\,, \quad i=1,2,\dots,n\,,\quad \mathfrak{X}(0) = x\,.
\end{equation}

Apply \textsc{It\^o}'s Rule to $f_2(x,y_1,\dots,y_n):=\varphi\big(T-x,(y_1,\dots,y_n)\big)$ with \eqref{eq:SDE3}:
\begin{equation}\label{eq:phi}\ 
\mathrm{d}\phi(t) =
\left[-\varphi_t \, \mathrm{d}t + \sum_i D_i \varphi\left(\mathfrak{b}_i \, \mathrm{d}t + \sum_k \mathfrak{s}_{ik}\, \mathrm{d}W_k(t)\right)\right.  
\left.+ \frac{1}{\,2\,} \sum_{i,j} D^2_{ij}\varphi \sum_k \mathfrak{s}_{ik}\mathfrak{s}_{jk}\, \mathrm{d}t\right](T-t,t,\mathfrak{X})\,,
\end{equation}
where for convenience, throughout the paper the values of $L$, $\phi$, $\mathfrak{b}_i$, $\mathfrak{s}_{ik}$, $\vartheta$, $\varphi_t$, $D_i\varphi$ and $D^2_{ij}\varphi$ at $(T-t,t,\mathfrak{X})$ stand for $L(t)$, $\phi(t)$, $\mathfrak{b}_i(t,\mathfrak{X})$, $\mathfrak{s}_{ik}(t,\mathfrak{X})$, $\vartheta(t,\mathfrak{X})$, $\varphi_t (T-t,\mathfrak{X}(t))$, $D_i\varphi(T-t,\mathfrak{X}(t))$ and $D^2_{ij}\varphi(T-t,\mathfrak{X}(t))$, respectively.

Summing (\ref{eq:SDE3}) over $i$ from $1$ to $n$ yields
\begin{equation}\label{eq:X}
\mathrm{d}X(t) = \left(\sum_i \mathfrak{b}_i\, \mathrm{d}t + \sum_{i,k} \mathfrak{s}_{ik}\, \mathrm{d}W_k(t)\right)(t,\mathfrak{X})\,.
\end{equation}

Finally, apply \textsc{It\^o}'s Rule to the exponential function for $L(\cdot)$:
\begin{equation}\label{eq:L}
\mathrm{d}L(t) = - L(t)\,\vartheta'(t,\mathfrak{X})\, \mathrm{d}W(t)\,.
\end{equation}

Plugging \eqref{eq:phi} -- \eqref{eq:L} into \textsc{It\^o}'s Rule for $f_3(r_1,r_2,r_3):=r_1 r_2 r_3$ gives
\begin{align}
\nonumber
\mathrm{d}(XL\phi)(t) 
=&\ \big[ L\phi\, \mathrm{d}X(t) + X\phi\, \mathrm{d}L(t) + XL \, \mathrm{d}\phi(t)
+ X \, \mathrm{d}\langle L,\phi\rangle_t + L \, \mathrm{d}\langle X,\phi\rangle_t + \phi \, \mathrm{d}\langle X,L\rangle_t\big] (t)\\
\label{eq:XLphi}
=&\  L\phi \left[\sum_i \mathfrak{b}_i\, \mathrm{d}t + \sum_{i,k} \mathfrak{s}_{ik}\, \mathrm{d}W_k(t)\right] 
- X\phi L \vartheta' \mathrm{d}W(t) 
- XL\varphi_t\, \mathrm{d}t \\
\nonumber
&+ XL\sum_i D_i \varphi\left[\mathfrak{b}_i\, \mathrm{d}t + \sum_k \mathfrak{s}_{ik}\, \mathrm{d}W_k(t)\right]
+ \frac{1}{\,2\,} XL\sum_{i,j} D^2_{ij}\varphi \sum_k \mathfrak{s}_{ik}\mathfrak{s}_{jk}\, \mathrm{d}t\\
\nonumber
&-\left. XL\sum_k \vartheta_k \sum_i D_i \varphi \mathfrak{s}_{ik}\, \mathrm{d}t 
+ L\sum_{k,j} \mathfrak{s}_{jk} \sum_i D_i \varphi \mathfrak{s}_{ik}\, \mathrm{d}t 
- \phi L\sum_{k,i} \mathfrak{s}_{ik}\vartheta_k\, \mathrm{d}t \right|_{(T-t,t,\mathfrak{X})}.
\end{align}
Rearranging (\ref{eq:XLphi}), we obtain
\begin{align*}
\mathrm{d}(XL\phi)(t) 
= &- XL\left(\phi_t - \frac{1}{\,2\,} \sum_{i,j} D^2_{ij}\varphi \sum_k \mathfrak{s}_{ik}\mathfrak{s}_{jk} - \frac{1}{X} \sum_{k,j} \mathfrak{s}_{jk} \sum_i D_i \varphi \mathfrak{s}_{ik} \right) \mathrm{d}t
+ L\phi\sum_{i,k} \mathfrak{s}_{ik}\, \mathrm{d}W_k(t) \\
&- X\phi L\vartheta' \mathrm{d}W(t)
+ XL\sum_i D_i \varphi \sum_k \mathfrak{s}_{ik}\, \mathrm{d}W_k(t)
+ L\phi \left(\sum_i \mathfrak{b}_i - \sum_{k,i} \mathfrak{s}_{ik}\vartheta_k\right)  \mathrm{d}t\\
&+ \left. XL \left(\sum_i D_i \varphi \mathfrak{b}_i - \sum_k \vartheta_k \sum_i D_i \varphi \mathfrak{s}_{ik} \right) \mathrm{d}t\right|_{(T-t,t,\mathfrak{X})}\\
=\ &\left. - XLg\,\mathrm{d}t 
+ L\phi\sum_{i,k} \mathfrak{s}_{ik} \, \mathrm{d}W_k(t) 
- X\phi L\vartheta' \mathrm{d}W(t)  
+ XL\sum_i D_i \varphi \sum_k \mathfrak{s}_{ik} \, \mathrm{d}W_k(t)\right|_{(T-t,t,\mathfrak{X})}\\
=\ &\left. - XLg\,\mathrm{d}t 
- X\phi L\vartheta'  \mathrm{d}W(t)  
+ L\sum_{i,k} \mathfrak{s}_{ik}\, \mathrm{d}W_k(t) \left(\phi+X D_i \varphi\right)\right|_{(T-t,t,\mathfrak{X})},
\end{align*}
where we used the definition \eqref{eq:hatG} of $g$ and the fact that $\,\mathfrak{b}=\mathfrak{s}\,\vartheta\,$.
\end{proof}

\section{An Alternative Proof for Theorem \ref{thm:viscosity1}  }
\label{section:subsol2}

We present here an alternative proof for Theorem \ref{thm:viscosity1}. We  still argue by contradiction, but avoid introducing the stopping time $\lambda$ of \eqref{eq:rho} and thus also the stopping time $\rho$ and  the constant $C_1$\,. We   also avoid using Lemma \ref{lemma:rho=nu}; instead, we   provide a lower bound for $\mathbb{E}[L(\nu)]$ in \eqref{eq:E[L]} below. The goal is to prove (\ref{eq:E[LXphi]}) for $\nu$ instead of $\rho\,$. We shall approximate $\nu$ by a sequence of stopping times $\nu_{\ell}$ for which (\ref{eq:E[LXphi]}) holds, then apply \textsc{Fatou}'s Lemma. This approach can  also be applied to the proof in Section \ref{section:supersol} for the supersolution property.

\smallskip
\begin{proof} 
According to Definition \ref{def:viscosity sol}\,(i) of viscosity subsolution with the $F$ in \eqref{eq:ourF}, it suffices to show that for any test function $\varphi \in C^{1,2}\left((0,\infty) \times \mathbb{R}_{+}^n\right)$ and  $(t_0,x_0) \in (0,\infty) \times \mathbb{R}_{+}^n$ with 
\begin{equation}\label{eq:max2}
\big({\widehat{\Phi}}^* - \varphi\big)(t_0,x_0) = 0 > \big({\widehat{\Phi}}^* - \varphi\big)(t,x)\,, \ \ \forall\  (t,x)\in (0,\infty) \times \mathbb{R}_{+}^n\,,
\end{equation}
(i.e., such that $(t_0,x_0)$ is a strict maximum of $\, {\widehat{\Phi}}^* - \varphi $), we have  $$\big(\varphi_t-\widehat{\mathcal{L}}\varphi\big)(t_0,x_0) \le 0\,.$$
Here $\widehat{\mathcal{L}}$ is defined in \eqref{eq:hatL}, and $\, \widehat{\Phi}^*\,$ is the upper-semicontinuous envelope of  $\, \widehat{\Phi}\,$ as in the definition \eqref{eq:u^*}. {\it We shall argue this by contradiction, assuming that} 
\begin{equation}\label{eq:hatG2}
\widehat{\mathcal{G}}(t_0,x_0)>0\, \qquad \text{for the function} \qquad \widehat{\mathcal{G}}(t,x) := \big(\varphi_t-\widehat{\mathcal{L}}\varphi\big)(t,x)\,.
\end{equation}
Since the function $F$ of \eqref{eq:F}  is continuous, so is the function $\widehat{\mathcal{G}}$ just introduced in \eqref{eq:hatG2}. 
There will
exist then, under this hypothesis and Assumption \ref{asmp1}, a neighborhood  $\,\mathcal{D}_{\delta} := (t_0 - \delta, t_0 + \delta) \times B_{\delta}(x_0)$ of $(t_0,x_0)$ in $(0,\infty) \times \mathbb{R}_{+}^n$ with $0< \delta < ||x_0||_1/ n$\,, on which    $\mathcal{K} (\cdot)$ is bounded and $\widehat{\mathcal{G}}(\cdot,\cdot)>0$ holds.

\smallskip

Let $C$ be a constant such that $\varphi(t,x)$, $||\theta||$, $|a_{ij}|<C$ $(1\le i,j \le n)$ hold for all pairs $(\theta, a=(a_{ij})_{n \times n})\in \mathcal{K}(x)$ and all $(t,x)\in \mathcal{D}_{\delta}$\,. We can assume that 
\begin{equation}\label{eq:deltaC}
16\, \delta C^2+2\, \delta^2 C^4 < 1 / 2
\end{equation}
by selecting a sufficiently small $\delta >0$\,. We notice that $|x_i-(x_0)_i|\le |x-x_0|<\delta$ holds for any $x=(x_1,\dots,x_n)\in \mathcal{D}_{\delta}$, thus
\begin{equation}\label{eq:||x||_12}
0 < ||x_0||_1 - n\delta < ||x||_1< ||x_0||_1 + n\delta \,,
\end{equation}
and introduce the  constants 
\begin{equation}\label{eq:C^star_3}
C_2:= - \max_{\partial \mathcal{D}_{\delta}} \, \big({\widehat{\Phi}}^* - \varphi \big)(t,x) \quad  {\mathrm{and}}\quad
C^{\star}_3:=\frac{C_2(||x_0||_1 - n\delta)}{2(||x_0||_1 + n\delta)}\left(\frac{1}{\,2\,}-16\, \delta C^2-2\, \delta^2 C^4\right),
\end{equation}
which are strictly positive by \eqref{eq:max2} and \eqref{eq:deltaC}, respectively.
We observe that 
$$\limsup_{(t,x)\to (t_0,x_0)}({\widehat{\Phi}}-\varphi)(t,x) = ({\widehat{\Phi}}^*-\varphi)(t_0,x_0)=0\,,$$ 
hence there exists $(t^*,x^*)\in \mathcal{D}_{\delta}$ such that  
\begin{equation}\label{eq:t*,x*2}
({\widehat{\Phi}}-\varphi)(t^*,x^*)>-C^{\star}_3\,;
\end{equation}
and by the definition \eqref{eq:Phi} of $\widehat{\Phi}$, there exists an admissible system $\mathcal{M}^{x^*}\in \widehat{\mathfrak{M}}(x^*)$ such that 
\begin{equation}\label{eq:u_M^x*2}
\mathfrak{u}_{\mathcal{M}^{x^*}}(t^*,x^*) 
> \widehat{\Phi}(t^*,x^*) - C^{\star}_3>\varphi(t^*,x^*) - 2\,C^{\star}_3\,, \quad{\mathrm{by}}\ \eqref{eq:t*,x*2}. 
\end{equation}
The remaining discussion in this section will be carried out under this admissible system, unless otherwise specified.
 
\medskip
\noindent
$\bullet~$ Let us start by constructing stopping times
\begin{equation}\label{eq:nu2}\ \ 
\nu \,(=\nu(\omega)):=\inf \big\{s\in (0, t^*] : \big(t^*-s,\mathfrak{X}(s)\big)\notin \mathcal{D}_{\delta} \big\} \le t^*-(t_0-\delta)=(t^*-t_0)+\delta< t^*\wedge 2\delta
\end{equation}
(by the definitions of $\mathcal{D}_{\delta}$ and $ t^*$), and for $\ell=1,2,\dots$,
\begin{equation}\label{eq:lambda_k}\ \ 
\lambda_{\ell} \,(=\lambda_{\ell} \,(\omega)):=\inf\{s>0: |\log L(s)| > \ell\}\uparrow \infty\,, \   \nu_{\ell} \,(=\nu_{\ell}\,(\omega)):=\nu \wedge \lambda_{\ell} \uparrow \nu\,,\ \mathbb{P}\mathrm{-a.s.}\ \mathrm{as}\ {\ell}\uparrow \infty
\end{equation}
with the usual convention inf$\,\emptyset = \infty$\,.

From definitions \eqref{eq:hatG2} and   \eqref{eq:hatL},
we see that  
\begin{equation}\label{eq:g2}
g(t,s,\mathfrak{X}) := (\varphi_t-\mathcal{L}_{\alpha(s,\mathfrak{X})}\varphi)\big(t,\mathfrak{X}(s)\big)\ge  \widehat{\mathcal{G}}\big(t,\mathfrak{X}(s)\big)\,,\quad \forall\ (t,s) \in (0, \infty)\times [0, \infty)\,.
\end{equation}
Recall that $\widehat{\mathcal{G}} (\cdot\,, \cdot) >0$ on $\mathcal{D}_{\delta}$ from the discussion right below \eqref{eq:hatG2}. Combining with \eqref{eq:g2} leads to
\begin{equation}\label{eq:g>02}
g(t^*-s,s,\mathfrak{X}) >0\,,\quad \forall\ s \in [0, \nu)\,.
\end{equation}

Let us apply now Lemma \ref{lemma:Ito} with $T=t^*$, integrating \eqref{eq:Ito} with respect to $t$ over $[0, \nu_{\ell}]$  and taking the expectation under $\mathbb{P}$, to obtain
\begin{equation}\label{eq:E[LXphi]2}
||x^*||_1\,\varphi(t^*,x^*) 
- \mathbb{E} \left[L(\nu_{\ell})X(\nu_{\ell})\varphi \big(t^*-\nu_{\ell},\mathfrak{X}(\nu_{\ell})\big)\right]
=\mathbb{E} \left[\int_0^{\nu_{\ell}} L(s)X(s)g(t^*-s,s,\mathfrak{X})\, \mathrm{d}s \right]> 0\,.
\end{equation}
Here, the strict inequality comes from \eqref{eq:g>02} and  the positivity of $\nu_{\ell}$\,; whereas,  in the equality, the expectations  of the integrals with respect to $\mathrm{d}W(t)$ or $\mathrm{d}W_k(t)$ have all vanished -- due to the  boundedness of the processes $\mathfrak{X}(\cdot)$ and $L(\cdot)$ on $[0,\nu_{\ell}]$, of the functions $\varphi$ and $D_i \varphi$ on $\overline{\mathcal{D}_{\delta}}\,$, and of the functionals $\vartheta(\cdot, \mathfrak{X})$, $\alpha_{ij}(\cdot, \mathfrak{X})$ (by Assumption \ref{asmp1}) and thus $\sigma_{ik}(\cdot, \mathfrak{X})$ on $[0,\nu_{\ell}]$.

\medskip
(We have made use here of the following facts.  
The eigenvalues $e_i$ of $\alpha$ are the nonnegative roots of the characteristic polynomial of $\alpha$, which is determined by the entries $\alpha_{ij}$\,; since the $\alpha_{ij}(\cdot, \mathfrak{X})$'s are bounded on  $[0,\nu_{\ell}]$, so are the $e_i$'s. Thus $\sigma$, which can be written as $\mathbf{QD}$, for some $n\times n$ orthonormal matrix $\mathbf{Q}$ and diagonal matrix $\mathbf{D}$ with diagonal entries $\sqrt{e_i}$\,, is also bounded.)

\medskip
Since almost surely $\nu_{\ell}\uparrow \nu$ ((\ref{eq:lambda_k})) and $L(\nu_{\ell})X(\nu_{\ell})\varphi \big(t^*-\nu_{\ell},\mathfrak{X}(\nu_{\ell}))>0$ for all $\ell$ (the positivity of $\varphi$ follows from \eqref{eq:max2} and \eqref{eq:u>=Phi}), \textsc{Fatou}'s Lemma gives
\begin{eqnarray} 
\label{eq:E[LXphi(nu)]}
&&\mathbb{E} \left[L(\nu)X(\nu)\varphi \big(t^*-\nu,\mathfrak{X}(\nu))\right] 
= \mathbb{E} \left[\liminf_{\ell\to \infty} L(\nu_{\ell})X(\nu_{\ell})\varphi \big(t^*-\nu_{\ell},\mathfrak{X}(\nu_{\ell}))\right]\\
\nonumber
&\le&\liminf_{\ell\to \infty} \mathbb{E}\left[ L(\nu_{\ell})X(\nu_{\ell})\varphi \big(t^*-\nu_{\ell},\mathfrak{X}(\nu_{\ell}))\right]
\le ||x^*||_1\,\varphi(t^*,x^*)\,,\quad  {\mathrm{by}}\ \eqref{eq:E[LXphi]2}\,.
\end{eqnarray}

\smallskip
Notice that $\,\big(t^*-\nu,\mathfrak{X}(\nu)\big) \in \partial\mathcal{D}_{\delta}\,$ holds by the definition \eqref{eq:nu2} of $\nu$, so we have 
\begin{equation}\label{eq:partialD2app}
\varphi \big(t^*-\nu,\mathfrak{X}(\nu)\big)
\ge {\widehat{\Phi}}^* \big(t^*-\nu,\mathfrak{X}(\nu)\big) +C_2
\ge {\widehat{\Phi}} \big(t^*-\nu,\mathfrak{X}(\nu)\big) +C_2\,.
\end{equation}
Plugging \eqref{eq:u_M^x*2} and (\ref{eq:partialD2app}) into (\ref{eq:E[LXphi(nu)]}) yields  
\begin{eqnarray}
\nonumber
0 &<&||x^*||_1\, \big[\,2\,C^{\star}_3   + \mathfrak{u}_{\mathcal{M}^{x^*}}(t^*,x^*)\big]
- \mathbb{E} \left[L(\nu)X(\nu)\left({\widehat{\Phi}}\big(t^*-\nu,\mathfrak{X}(\nu)\big)+C_2\right)\right] \\
\nonumber
&=&||x^*||_1\,\mathfrak{u}_{\mathcal{M}^{x^*}}(t^*,x^*)
- \mathbb{E} \left[L(\nu)X(\nu){\widehat{\Phi}} \big(t^*-\nu,\mathfrak{X}(\nu)\big)\right]+2\,C^{\star}_3\,||x^*||_1- C_2\, \mathbb{E} \big[L(\nu)X(\nu)\big]\\
\nonumber
&\le& 2\,C^{\star}_3\,||x^*||_1- C_2\, \mathbb{E} \big[L(\nu)X(\nu)\big] 
\ <\ 2\,C^{\star}_3(||x_0||_1 + n\delta)- C_2(||x_0||_1 - n\delta)\, \mathbb{E} \big[L(\nu)\big]\,,
\end{eqnarray}
(we have used Proposition \ref{prop:martingale} in the third step and the last inequality of \eqref{eq:||x||_12} at last).

Recall the definition \eqref{eq:C^star_3} of $C^{\star}_3$\,. We will arrive at a contradiction and hence complete the argument, as soon as we have shown the following inequality:
\begin{equation}\label{eq:E[L]}
\mathbb{E} \big[L(\nu)\big] \ge \frac{1}{\,2\,} -  16\, \delta C^2 - 2\, \delta^2 C^4.
\end{equation} 
 (This explains why we constructed $C^{\star}_3$ as we did in (\ref{eq:C^star_3}); in fact, setting $C^{\star}_3$ to be any value less than the right-hand side of (\ref{eq:C^star_3}) would also work). First, we observe the following double inequality
\begin{equation}\label{eq:e^r}
e^r \ge \frac{3}{2e} - \frac{r^2}{2e} > \frac{1}{\,2\,} - r^2,\quad \forall\ r\in \mathbb{R}\,.
\end{equation} 
The second inequality is obvious since $2<e<3$\,. For the first inequality, we set 
\begin{equation*}
f(r):=e^r - \frac{3}{2e} + \frac{r^2}{2e}
\end{equation*} 
and find that $f(-1)=0$\,, $f'(-1)=0$ and $f''(r)>0$\,. Hence $f(r)$ achieves its minimum 0 at $r=-1$\,. Applying \eqref{eq:e^r} to $\log L(\nu)$ yields
$$\mathbb{E} \big[L(\nu)\big]
\ge \frac{1}{\,2\,} - \mathbb{E} \left[\big(\log L(\nu)\big)^2\right].$$
Therefore, it suffices to show that 
\begin{equation}\label{eq:E[sup lgL]}
\mathbb{E}\left[\sup_{0\le t\le \nu} \big( \log  L (t) \big)^2\right] 
\le  16\, \delta C^2 + 2\, \delta^2 C^4.
\end{equation}

For any $t\in(0, \nu]\,,$ we have 
\begin{eqnarray*}
  \big( \log  L (t) \big)^2 &=& \left|-\int_0^t\vartheta'(s,\mathfrak{X})\, \mathrm{d}W(s) - \int_0^t\frac{1}{\,2\,} \,\big|\big|\vartheta(s,\mathfrak{X})\big|\big|^2\, \mathrm{d}s\right|^2\\
&\le& 2\left|\int_0^t\vartheta'(s,\mathfrak{X})\, \mathrm{d}W(s)\right|^2
+ 2\left|\int_0^t\frac{1}{\,2\,} \,\big|\big|\vartheta(s,\mathfrak{X})\big|\big|^2\, \mathrm{d}s\right|^2
\end{eqnarray*}
It follows from $t\le \nu<2\delta$ that
\begin{equation*}
\int_0^t \frac{1}{\,2\,} \,\big|\big|\vartheta(s,\mathfrak{X})\big|\big|^2\, \mathrm{d}s
\, \le \, \frac{t}{\,2\,} \,C^2 \, 
\le\, \delta C^2,
\end{equation*}
and therefore
\begin{equation*}
\mathbb{E}\left[\sup_{0\le t\le \nu} \big( \log  L (t) \big)^2\right] 
\le 2\, \mathbb{E}\left[\sup_{0\le t\le \nu} \left|\int_0^t\vartheta'(s,\mathfrak{X})\, \mathrm{d}W(s)\right|^2\right]
+ 2\, \delta^2 C^4.
\end{equation*}
Finally, the \textsc{Burkholder-Davis-Gundy} Inequality gives
\begin{eqnarray*}
\mathbb{E}\left[\sup_{0\le t\le \nu} \left|\int_0^t\vartheta'(s,\mathfrak{X})\, \mathrm{d}W(s)\right|^2\right]
\le 4\, \mathbb{E} \left[\int_0^{\nu}\left|\left|\vartheta'(s,\mathfrak{X})\right|\right|^2\mathrm{d}s\right]\le 4\, \mathbb{E} \left[\nu C^2\right]
\le 8\, \delta C^2,
\end{eqnarray*}
and \eqref{eq:E[sup lgL]} follows. 
\end{proof}

 \bigskip


\end{document}